\documentclass[11pt]{amsart}
\usepackage[utf8]{inputenc}
\usepackage[pdftex]{graphicx}
\usepackage{amssymb}
\usepackage{amsmath}
\usepackage{amsfonts}
\usepackage{amsthm}
\usepackage{bm}
\usepackage{mathrsfs}
\usepackage{appendix}
\usepackage{enumerate}
\usepackage[left=3cm, right=3cm, bottom=3.4cm]{geometry}
\usepackage[stretch=30,shrink=30]{microtype}
\usepackage{stmaryrd}
\usepackage[normalem]{ulem} 

\usepackage{xcolor}
\definecolor{antiquefuchsia}{rgb}{0.57, 0.36, 0.51}
\definecolor{azure}{rgb}{0.0, 0.5, 1.0}

\usepackage[backref=page, colorlinks = true, linkcolor = black, citecolor = antiquefuchsia]{hyperref}

\renewcommand*{\backref}[1]{}
\renewcommand*{\backrefalt}[4]{%
    \ifcase #1 (Not cited.)%
    \or        (Cited on page~#2.)%
    \else      (Cited on pages~#2.)%
    \fi}

\makeatletter
\def\th@plain{%
	\thm@notefont{}% same as heading font
	\itshape % body font
}
\def\th@definition{%
	\thm@notefont{}% same as heading font
	\normalfont % body font
}
\makeatother

\numberwithin{equation}{section}

\newtheorem{theorem}{Theorem}[section]
\newtheorem{lemma}[theorem]{Lemma}
\newtheorem{proposition}[theorem]{Proposition}

\theoremstyle{definition}
\newtheorem{definition}[theorem]{Definition}

\theoremstyle{remark}
\newtheorem{remark}[theorem]{Remark}

\newcommand{\N}{\mathbb{N}}

\newcommand{\R}{\mathbb{R}}

\DeclareMathOperator{\dist}{dist}

\DeclareMathOperator{\vol}{vol}

\DeclareMathOperator{\Sing}{Sing}

\usepackage[capitalise,nameinlink]{cleveref}
\crefformat{equation}{#2(#1)#3}
\crefrangeformat{equation}{#3(#1)#4 to~#5(#2)#6}
\crefmultiformat{equation}{#2(#1)#3}{ and~#2(#1)#3}{, #2(#1)#3}{ and~#2(#1)#3}

\newcommand{\eps}{\varepsilon}

\title[Symmetric Log-epiperimetric Inequality for Harmonic Maps]{Symmetric Log-epiperimetric Inequality for Harmonic Maps with Analytic Target and Applications} 

\author[Riccardo Caniato]{Riccardo Caniato}
\address{California Institute of Technology, Department of Mathematics, 1200 E California Blvd, MC 253-37, CA 91125, United States of America}
\email{rcaniato@caltech.edu}

\author[Davide Parise]{Davide Parise}
\address{University of California San Diego, Department of Mathematics, 9500 Gilman Drive \#0112, La Jolla, CA 92093-0112, United States of America}
\email{dparise@ucsd.edu}

\date{\today} 

\begin{document}
\begin{abstract}
    We establish a direct symmetric (log)-epiperimetric inequality for harmonic maps with analytic target and we leverage on this result to achieve a new proof of Simon's celebrated uniqueness of tangents with isolated singularity for energy minimizing harmonic maps. Moreover, we show that tangents at infinity of energy minimizing harmonic maps with suitably controlled energy growth are always unique, by exploiting the lower bound entailed in the symmetric (log)-epiperimetric inequality.
\end{abstract}
\maketitle
\tableofcontents
\section{Introduction}

\subsection*{Epiperimetric inequalities in geometric variational problems}
Given a geometric functional $\mathscr{F}$ on a domain with boundary---interpreted as the ``energy'' of the problem---epiperimetric type inequalities provide a precise quantitative estimate on the energy sub-optimality of the natural homogeneous extension of a given boundary datum inside the domain. This information allows to obtain rates of decay, or growth, for the energy density of $\mathscr{F}$-minimizers, thus providing a suitable framework to study the uniqueness of their tangents either at a point or at infinity. 

Introduced by Reifenberg in his seminal work \cite{Reifenberg}, epiperimetric inequalities have played a central role in the study of minimal surfaces. In particular, they were originally developed to establish the analyticity of solutions to the Plateau problem, in the form posed in \cite{ReifenbergAnalytic}. Later on, White exploited this technique in \cite{white-unique} to prove uniqueness of tangent cones for two-dimensional area-minimizing integral currents without boundary in $\mathbb{R}^n$. Building on these ideas, Rivière introduced in \cite{Riviere} the concept of a \textit{lower} epiperimetric inequality, thereby establishing a rate of decay for the decreasing density of an area-minimizing integral $2$-cycle in $\mathbb{R}^n$. One of the key insights arising from Rivière’s work was the identification of the \textit{splitting before tilting} phenomenon, a concept that proved essential in his joint regularity results with Tian for $1-1$ integral cycles \cite{RiviereTian}. This notion has since been extensively applied in the regularity theory developed by De Lellis, Spadaro, and Spolaor, who extended these techniques to area-minimizing and semicalibrated currents (see \cite{CenterManifold, BranchedCenterManifold, Spolaor}). More recently, in the context of two-dimensional almost minimal currents, these authors established a new epiperimetric inequality \cite{DeLellisSpadaroSpolaor}, thereby generalizing White’s earlier results.

Subsequently, epiperimetric inequalities have attracted considerable attention in the setting of free boundary problems, in particular due to the efforts of Spolaor and Velichkov. For instance, in \cite{SpolaorVelichkov}, these two authors adapted \textit{direct} epiperimetric inequalities for various free boundary problems in two dimensions. 
%pioneered the adaptation of direct epiperimetric inequalities to this setting in \cite{SpolaorVelichkov}. 
Later, in collaboration with Colombo, the same authors proved a \textit{logarithmic} epiperimetric inequality for obstacle-type problems \cite{ColomboSpolaorVelichkov1, ColomboSpolaorVelichkov2}. Another significant advancement came in \cite{ESVduke}, where Engelstein, Spolaor, and Velichkov introduced a novel method of proof: by reducing the problem to an estimate for a functional defined on the sphere and analyzing its gradient flow, they offered a new perspective on proving such inequalities. While their focus was on the Alt-Caffarelli functional, this approach has since been instrumental in the analysis of multiplicity-one stationary cones with isolated singularities \cite{EngelsteinSpolaorVelichkov}, leading to new $\varepsilon$-regularity theorems for almost minimizers. Besides, the same method was exploited by the authors of the present paper to prove uniqueness of tangents for Yang--Mills connections with isolated singularities \cite{CaniatoParise}. For further developments along these lines, see also \cite{SpolaorVelichkovEng}. Notably, in \cite{SymmetricEdelenSpolaorVelichkov}, the authors introduced the notion of \textit{symmetric} logarithmic epiperimetric inequality for the Alt-Caffarelli functional and for almost minimizing currents, which entails at once a bi-later (upper and lower) control on the energy density. 

A unifying feature of these works is their reliance on \textit{direct} constructions of competitors—explicitly built functions that allow one to prove decay estimates around singular points. Such direct methods stand in contrast to the alternative approach of proving epiperimetric inequalities via contradiction, a strategy that typically employs linearization techniques. Contradiction-based proofs have been used effectively in various settings, albeit often limited to regular points or singularities under strong structural assumptions. Notable examples include Taylor’s work on area-minimizing flat chains modulo 3 and $(\mathbf{M}, \varepsilon, \delta)$-minimizers \cite{TaylorFlatChains, TaylorSoapBubble, TaylorEllipsoidal}. For free boundary problems, Weiss’s classical contribution \cite{Weiss} established epiperimetric inequalities at flat singular points and along the top stratum of the singular set in the obstacle problem. See also \cite{TaylorCapillarity} for the very first instance of an epiperimetric inequality for free boundary problems. For the thin obstacle problem, we refer to \cite{FocardiSpadaro, GarofaloPetrosyanMariana}, while in the context of harmonic measures, the work of Badger, Engelstein, and Toro \cite{Badger} inverstigated epiperimetric inequalities for functions that do not minimize any energy.  

\subsection*{Statement of the main results} 
Inspired by \cite{EngelsteinSpolaorVelichkov, ESVduke, SymmetricEdelenSpolaorVelichkov, CaniatoParise}, the aim of this article is to prove a symmetric log-epiperimetric inequality for harmonic maps with real-analytic target. To the best of the authors' knowledge this is the first instance of such an inequality in this geometric setting. More precisely, we prove the following--see Section \ref{sec: preliminaries} for the notations appearing in the statement.
\begin{theorem}\label{Theorem: (log)-epiperimetric inequality for harmonic maps}
	Let $N\subset\R^k$ be a closed real-analytic submanifold in $\R^k$ and let $n\in\N$ be such that $n\ge 3$. Let $u_0\in C^{\infty}(\mathbb{S}^{n-1},N)$ be a harmonic map on $\mathbb{S}^{n-1}$ and let $\tilde u_0\in W^{1,2}(\mathbb{B}^n,N)$ be its 0-homogeneous extension inside $\mathbb{B}^n$, given by
	\begin{align*}
		\tilde u_0:=\bigg(\frac{\cdot}{\lvert\,\cdot\,\rvert}\bigg)^*u_0.
	\end{align*}
	There exist constant $\varepsilon, \delta > 0$, and $\gamma \in [0, 1)$ depending on the dimension and $u_0$ such that the following holds. If $u\in C^{2,\alpha}(\mathbb{S}^{n-1},N)$ is such that
	\begin{align*}
		\|u-u_0\|_{C^{2,\alpha}(\mathbb{S}^{n-1})}<\delta
	\end{align*}
	then there exists $\hat u\in W^{1,2}(\mathbb{B}^n,N)$ such that $\hat  u|_{\mathbb{S}^{n-1}}=u$ and 
	\begin{align} \label{equation: main epiperimetric inequality}
		\mathscr{E}_{\mathbb{B}^n}(\hat u\,;\tilde u_0)\le\mathscr{E}_{\mathbb{B}^n}(\tilde u\,;\tilde u_0)-\varepsilon\lvert\mathscr{E}_{\mathbb{B}^n}(\tilde u\,;\tilde u_0)\rvert^{1+\gamma},
	\end{align}
	where $\tilde u\in W^{1,2}(\mathbb{B}^n,N)$ is the $0$-homogeneous extension of $u$ inside $\mathbb{B}^n$. If the kernel of the second variation is integrable, we can take $\gamma = 0$. 
\end{theorem}

\begin{remark} \label{remark: two proofs}
    We give two different proofs of Theorem \ref{Theorem: (log)-epiperimetric inequality for harmonic maps}. The first one follows the strategies outlined in \cite{EngelsteinSpolaorVelichkov, ESVduke, SymmetricEdelenSpolaorVelichkov, CaniatoParise}, and occupies the main body of this article. We will outline it more carefully later in this introduction. On the other hand, the second one can be found in Appendix \ref{appendix: EP ineq and LS ineq}, and uses parabolic variational inequalities to construct the competitor, and consequently does not require to split the trace into different components and deal with them separately. Besides, it sheds some light on the relation between the (log)-epiperimetric and \L ojasiewicz--Simon inequalities by formalising that the latter implies the former under certain hypothesis. This was highlighted in \cite[Section 3]{colombo-spolaor-velichkov-3}, and we refer the reader to it for a more detailed discussion. In particular, the present paper complements the extensive literature on harmonic maps and \L ojasiewicz inequalities. We conclude by mentioning that a similar issue was emphasized in \cite[Remark 1.8]{CaniatoParise}. 
\end{remark}
\begin{remark}
	Note that in the simple setting of \textit{harmonic functions}, i.e. when $N = \mathbb{R}$, one can prove an epiperimetric inequality for both the Dirichlet energy, with respect to $0$-homogeneous extensions as in Theorem \ref{Theorem: (log)-epiperimetric inequality for harmonic maps}, and the Weiss energy, with respect to $\alpha$-homogeneous extensions. More precisely, in the second setting, for $f \in W_{\mathrm{loc}}^{1, 2}(\mathbb{R}^n)$, a point $x \in \mathbb{R}^n$, and $\alpha \in (0, \infty)$, we define the \textit{Weiss energy} to be 
	\begin{align*}
		W_\alpha(r, x; f) & := \frac{1}{r^{n - 2 + 2 \alpha}}\int_{B_r(x)} \vert \nabla f \vert^2 - \frac{\alpha}{r^{n - 1 + 2\alpha}} \int_{\partial B_r(x)} f^2. 
	\end{align*}
    A quick computation reveals that this quantity is monotone for harmonic functions and by expanding in spherical harmonics one can prove the following: for every $n \geq 2$, $\alpha > 0$, there exists $\varepsilon \in [0, 1)$ with the following significance. For $u \in W^{1, 2}(B_r(x))$ homogeneous of degree $\alpha$ about $x$, and $f$ the harmonic extension of $u \vert_{\partial B_r(x)}$ to $B_r(x)$, i.e. the trace of $u$ to the boundary, we have 
	\begin{equation*}
		W_\alpha(r, x; f) \leq (1 - \varepsilon) W_\alpha(r, x; u). 
	\end{equation*} See \cite[Appendix A]{Badger} for a proof of this result.  
\end{remark}
We note that Theorem \ref{Theorem: (log)-epiperimetric inequality for harmonic maps} is purely variational, and does not depend on any underlying PDE. As consequence of it, we obtain an alternative, albeit similar, proof of the celebrated uniqueness of tangent maps with isolated singularities due to Simon \cite{AsymptoticSimon}. See also \cite[Section 2.5]{LinWangBookHM}, and \cite[Section 3.10]{simon-book} for very elegant proofs of the following theorem. 
\begin{theorem} \label{Theorem: uniqueness of tangent maps with isolated singularities}
	Let $N\subset\R^k$ be a closed real-analytic submanifold in $\R^k$, and $\Omega \subset \mathbb{R}^n$ be an open set, with $n\ge 3$. Consider $u \in W^{1, 2}(\Omega, N)$ an energy minimizing harmonic map. Suppose that $\varphi$ is a tangent map at $u$ at some point $y \in \Sing(u)$, and assume that $\Sing(\varphi) = 0$. Then, $\varphi$ is the unique tangent map for $u$ at $y$. Furthermore, we have the expansion 
	\begin{equation*}
		u(y + r\omega) = \varphi(\omega) + \epsilon(r, \omega), 
	\end{equation*}
	for $\omega \in S^{n - 1}$, and where the error term $\epsilon$ satisfies the following asymptotics 
	\begin{equation*}
		\limsup_{r \rightarrow 0} \vert \log(r) \vert^\alpha \sup_{\omega \in S^{n - 1}} \vert \epsilon(r, \omega) \vert = 0, 
	\end{equation*}
	for some $\alpha > 0$. 
\end{theorem}
\begin{remark}
	Theorem \ref{Theorem: uniqueness of tangent maps with isolated singularities} is sharp in the following senses. First, the logarithmic decay is the best possible in light of examples constructed in \cite{AdamsSimon} and \cite{GulliverWhite}. On the other hand, the analyticity hypothesis on the target $N$ cannot be replaced by smoothness in view of \cite{WhiteNonunique}. %We refer the reader to the aforementioned references, as well as to \cite[Section 3.10]{simon-book} for further details on this. 
\end{remark}
Being a direct consequence of Theorem \ref{Theorem: (log)-epiperimetric inequality for harmonic maps}, our proof of Theorem \ref{Theorem: uniqueness of tangent maps with isolated singularities} is again purely variational, contrarily to Simon's one. Indeed, the arguments in \cite{AsymptoticSimon} are based on a careful analysis of the convergence properties of solutions to certain parabolic PDEs, and on growth estimates for the corresponding solutions. This also suggests that the argument used here would carry over for \textit{almost} minimizers, although this notion does not seem to have ever appeared in the literature. Moreover, the lower bound entailed in \eqref{equation: main epiperimetric inequality} immediately gives the following uniqueness of tangents at infinity for energy minimizing harmonic maps with analytic target in any dimension---see Section \ref{sec: uniqueness of tangents at infinity} for a proof. 
\begin{theorem}\label{thm: uniqueness of blow-downs}
    Let $N\subset\R^k$ be a closed real-analytic submanifold in $\R^k$. Let $u \in W_{\mathrm{loc}}^{1, 2}(\R^n, N)$ be an energy minimizing harmonic map such that
    \begin{align*}
        \frac{1}{\rho^{n-2}}\int_{B_{\rho}(0)}\lvert du\rvert^2\, d\mathcal{L}^n\le \Lambda \qquad\forall\,\rho\in(0,+\infty).
    \end{align*}
    for some $\Lambda>0$. Assume that $u$ has a tangent map $\varphi$ at infinity such that $\operatorname{Sing}(\varphi)\subset\{0\}$. Then, $\varphi$ is the unique tangent map of $u$ at infinity and there exists $\alpha>0$ such that
    \begin{align*}
        \|u(\rho\,\cdot\,)-\varphi\|_{L^2(\mathbb{S}^{n-1})}\le C\log(\rho)^{\alpha} \qquad\forall\,\rho\in(1,+\infty).
    \end{align*}
\end{theorem}
The last few decades have seen remarkable progress on the understanding of the singular set of harmonic maps, especially when it comes to its optimal bound.  
In supercritical dimension $n\ge 3$, the sharp bound was established for energy minimizing harmonic maps by Schoen and Uhlenbeck \cite{schoen-uhlenbeck}, while for stationary harmonic maps, the best available bound on the size of their singular set is given by the subsequent contributions of \cite{evans}, \cite{bethuel} and \cite{riviere-struwe}. If we drop the stationarity assumption, it is known by work of Rivi\`ere in \cite{riviere-everywhere-discontinuous} that harmonic maps on $n$-manifolds with $n\ge 3$ can be even everywhere discontinuous and, therefore, no regularity theory is possible in this setting.

On the other hand, understanding under which conditions the tangents to harmonic maps at their singular points are unique is a widely open problem and, so far, very little is known in full generality beyond non-isolated singularities. In the framework of stationary harmonic maps, first Rivi\`ere and Tian showed in \cite{riviere-tian} that strongly approximable pseudoholomorphic maps from a $4$-dimensional almost complex manifold into complex algebraic varieties have a unique tangent map at every point and established the optimal bound on the size of their singular set. Recently, the uniqueness of tangent maps was generalized by the first-named author and Rivi\`ere to arbitrary dimension of the domain in \cite{caniato-riviere}. Both these results suggest that uniqueness of tangents is quite a strong property and, therefore, may hold true just under some special geometric assumptions on the structure both of the map and of the target manifold. 

More broadly, this uniqueness problem pertains to other geometric PDEs, and it is a central question in geometric analysis. For instance, in the celebrated work \cite{AsymptoticSimon} that we mentioned above, Simon also proved a result similar to the one in Theorem \ref{Theorem: uniqueness of tangent maps with isolated singularities} for stationary varifolds. We refer the reader to the surveys \cite{de2022regularity, wickramasekera2014regularity}, and the references therein, for further details on the uniqueness of tangent cones problem in the setting of minimal submanifolds. On the other hand, in the framework of mean curvature flow, Colding and Minicozzi in \cite{ColdingMinicozzi} exploited an infinite dimensional \L ojasiewicz inequality to prove a uniqueness result at generic cylindrical singularities. For Einstein manifolds, the same authors proved in \cite{ColdingMinicozziEinstein} that the tangent cone at infinity of a Ricci-flat manifold with Euclidean volume growth is unique, provided one tangent cone at infinity has a smooth cross-section (compare this result with Theorem \ref{thm: uniqueness of blow-downs}). 

\subsection*{Ideas of the proof and structure of the article} 

The proof of Theorem \ref{Theorem: (log)-epiperimetric inequality for harmonic maps} follows the strategy outlined in \cite{EngelsteinSpolaorVelichkov, ESVduke, CaniatoParise} and it relies on constructing a competitor function with energy smaller than the one of the 0-homogeneous extension of $u_0$. We start by rewriting the energy energy discrepancy $\mathscr{E}_{\mathbb{B}^n}(\cdot\,;\cdot)$ in Theorem \ref{Theorem: (log)-epiperimetric inequality for harmonic maps} in a more convenient form. This is our slicing lemma, cf. Lemma \ref{Lemma: slicing lemma} and it is done in Section \ref{sec: slicing and LS reduction}. There we also recall the Lyapunov--Schmidt reduction adapted to the setting of harmonic maps, cf. Lemma \ref{Lyapunov-Schmidt reduction}. In particular, the slicing lemma suggests that we can construct the competitor by flowing inwards the slices of the trace $u_0$ to decrease the energy. The question then becomes how to choose the appropriate directions of the flow. To answer it we turn to the second variation, which can be written as a linear elliptic operator with compact resolvent. This last property is crucial as it implies that the linearized operator has a finite dimensional kernel, thus allowing us to decompose the datum $u_0$ as the sum of the projections on the kernel, the positive, and the negative eigenvalues, i.e. the index. 
The map $u_0$ being harmonic on the sphere $\mathbb{S}^{n - 1}$, positive directions will increase the energy to second order, while negative directions will decrease it. Thus, to construct the competitor, we wish to move move $u_0$ towards zero in the former, while increasing the contribution of the latter. To deal with the kernel we resort to a \textit{finite dimensional harmonic map heat flow} combined with a finite dimensional \L ojasiewicz inequality, cf. Lemma \ref{lemma: finite dimensional Lojasiewicz inequality} to make the estimate more quantitative. Note that this finite dimensional contribution is ultimately responsible for the logarithmic decay, instead of a polynomial one. See Section \ref{sec: proof of epiperimetric ineq} for a proof of Theorem \ref{Theorem: (log)-epiperimetric inequality for harmonic maps}. In the special situation in which the projection of $u_0$ on the kernel of the second variation vanishes, the so-called the integrable case, the proof simplifies significantly, see Subsection \ref{subsec: integrable case}. The adaptation of this proof to yet another geometric setting is a testament to its remarkable flexibility. Moreover, it establishes a precise connection between the kernel of the second variation and the logarithmic decay.

As an application of Theorem \ref{Theorem: (log)-epiperimetric inequality for harmonic maps}, we prove Theorem \ref{Theorem: uniqueness of tangent maps with isolated singularities}. The reader can find this in Section \ref{sec: proof of uniqueness}. We start by showing that the $L^2$-difference of a blow-up at comparable scales is small. The key idea is then to exploit the log-epiperimetric inequality to infer a bound for the energy density $\Theta(\rho, y; u)$ at all dyadic scales, which can then be converted to a bound at all scales. Uniqueness of the tangent map then follows by a standard Dini-type estimate that we include in Appendix \ref{sec: appendix criterion}. We conclude in Appendix \ref{appendix: EP ineq and LS ineq} by giving the alternative proof of Theorem \ref{Theorem: (log)-epiperimetric inequality for harmonic maps} already mentioned in Remark \ref{remark: two proofs}. More precisely, we construct the competitor  as in \cite[Proposition 3.1]{colombo-spolaor-velichkov-3} by flowing the full trace inwards via \textit{infinite dimensional harmonic map heat flow} instead of dealing with the different projected components. As before, the slicing lemma is the crucial first of this analysis. This greatly simplifies the proof, while having the drawback of relying on the infinite dimensional \L ojasiewicz-Simon inequality for harmonic maps established in \cite{AsymptoticSimon}. 

\subsection*{Acknowledgements}
The authors would like to thank Luca Spolaor for useful discussions. Part of this work was partially supported by the National Science Foundation Grant No. DMS-1928930, while the authors were in
residence at the Simons Laufer Mathematical Sciences Institute
(formerly MSRI) in Berkeley, California, during the Fall 2024
semester. D.P. acknowledges the support of the AMS-Simons travel grant.

\section{Preliminaries on harmonic maps} \label{sec: preliminaries}
The aim of this section is to collect basic definitions and properties on harmonic maps used later in the article. 
\begin{definition}[Dirichlet energy and harmonic maps]
    Let $N\subset\R^k$ be a closed, i.e. compact and boundaryless, smooth submanifold in $\R^k$. Let $(M^n,g)$ be a smooth $n$-dimensional Riemannian manifold and let $\Omega\subset M$ be an open subset of $M$. 
    \newline 
    An $N$-valued \textit{harmonic map} on $\Omega$ is a critical point $u\in W^{1,2}(\Omega,N)$\footnote{Here and throughout, we denote by $W^{1,2}(\Omega,N)$ the space of the functions $u\in W^{1,2}(\Omega,\R^k)$ such that $u(x)\in N$ for $\vol_g$-a.e. $x\in\Omega$.} of the functional $\mathscr{D}(\,\cdot\,,\Omega):W^{1,2}(\Omega,N)\to[0,+\infty)$ given by
    \begin{align*}
        \mathscr{D}_M(v\,;\Omega):=\int_{\Omega}\lvert dv\rvert_g^2\, d\vol_g\qquad\forall\, v\in W^{1,2}(\Omega,N).
    \end{align*}
    For short we denote $\mathscr{D}_M(\,\cdot\,\,;M)$ simply by $\mathscr{D}_M(\,\cdot\,)$. We call $\mathscr{D}_M(\,\cdot\,\,;\Omega)$ the \textit{Dirichlet energy} functional on $\Omega$. 
\end{definition}
Furthermore, for a given harmonic map $u$, we define the \textit{singular set} of $u$, denoted $\Sing(u)$, to be the complement of the set of points at which $u$ is smooth, i.e. the \textit{regular set}. 
\begin{definition} \label{def: discrepancy energy}
    Let $N\subset\R^k$ be a closed smooth submanifold in $\R^k$. Let $(M^n,g)$ be a smooth $n$-dimensional Riemannian manifold and let $\Omega\subset M$ be an open subset of $M$. Let $u_0\in W^{1,2}(\Omega,N)$. 
    \newline 
    The functional $\mathscr{E}_{M}(\,\cdot\,\,;u_0,\Omega):W^{1,2}(\Omega,N)\to\R$ given by
    \begin{align*}
        \mathscr{E}_M(u\,;u_0,\Omega):=\mathscr{D}_M(u\,;\Omega)-\mathscr{D}_M(u_0\,;\Omega) \qquad \forall u \in \,W^{1,2}(\Omega,N)
    \end{align*}
    is called \textit{Dirichlet energy discrepancy with respect to $u_0$} on $\Omega$.
\end{definition}
\begin{lemma}\label{Lemma: the analyticity of energy discrepancy in a neighbourhood of the origin}
    Let $N\subset\R^k$ be a closed real-analytic submanifold in $\R^k$ and let $\pi_N:W_{\delta_0}(N)\to N$ be the nearest-point projection from a tubular neighbourhood $W_{\delta_0}(N)$ of $N$ into $N$. Let $(M^n,g)$ be a smooth $n$-dimensional Riemannian manifold and let $\Omega\subset M$ be an open subset of $M$. Let $u_0\in C^{\infty}(\Omega,N)$ be a smooth map on $\Omega$. Then, the following facts hold.
    \begin{enumerate}[(i)]
        \item The functional $\mathscr{F}_{M}(\,\cdot\,\,;u_0,\Omega):\mathcal{B}_{\delta_0}(0)\subset C^{2,\alpha}(\Omega,u_0^*TN)$\footnote{Here and throughout, for every $k=0,...,+\infty$ we denote by by $C^k(\Omega,u_0^*TN)$ the space of the $C^k$ sections on $\Omega\subset M$ of the smooth vector bundle $u_0^*TN$ over $M$.}$\to\mathbb{R}$ given by 
        \begin{align*}
            \mathscr{F}_{M}(v\,;u_0,\Omega):=\mathscr{E}_{M}(\pi_N(u_0+v)\,;u_0,\Omega)\qquad\forall\, v\in \mathcal{B}_{\delta_0}(0)\subset C^{2,\alpha}(\Omega,u_0^*TN)
        \end{align*}
        is real-analytic on $\mathcal{B}_{\delta_0}(0)\subset C^{2,\alpha}(\Omega,u_0^*TN)$ in the sense of \cite[Section 3.13]{simon-book}.
        \item For every $u\in B_{\delta_0}(u_0)\subset C^2(\Omega,N)$
        there exists $\tilde u\in \mathcal{B}_{\delta_0}(0)\subset C^{2,\alpha}(\Omega,u_0^*TN)$ so that
        \begin{align*}
            \mathscr{E}_{M}(u\,;u_0,\Omega)=\mathscr{E}_{M}(\pi_N(u_0+\tilde u)\,;u_0,\Omega)=\mathscr{F}_{M}(\tilde u\,;u_0,\Omega).
        \end{align*}
    \end{enumerate}
    \begin{proof}
        We first prove (i). Fix any $v\in C^{2,\alpha}(\Omega,u_0^*TN)$ and notice that
        \begin{align*}
            d(\pi_N(u_0+v))&=d\pi_N(u_0+v)[du_0+dv].
        \end{align*}
        Hence, letting 
        \begin{align*}
            \mathcal{D}:=\{(x,z,\eta) \mbox{ : } x\in\Omega,\,z\in B_{\delta_0}^k(0),\,\eta\in T_x\Omega\otimes\R^k\}
        \end{align*}
        we define $F:\mathcal{D}\to\R$ as
        \begin{align*}
            F(x,z,\eta):=\big\lvert d\pi_N(u_0(x)+z)[du_0(x)+\eta]\big\rvert_g^2-\lvert du_0(x)\rvert_g^2 \qquad\forall\,(x,z,\eta)\in\mathcal{D}
        \end{align*}
        and we notice that
        \begin{align*}
            \mathscr{F}_M(v\,;u_0,\Omega)=\int_{\Omega}F(x,v(x),dv(x))\, d\vol_g(x) \qquad\forall\ v\in \mathcal{B}_{\delta_0}(0)\subset C^{2,\alpha}(\Omega,u_0^*TN).
        \end{align*}
        Notice that $F$ is smooth in the variable $x$, because both $u_0$ and $\pi_N$ are smooth functions. Moreover, since $\pi_N$ is real-analytic, $F$ and all its derivatives with respect to the variable $x$ are real-analytic with respect to the variables $z,\eta\in\mathcal{D}$. Hence, we conclude that $\mathscr{F}_M(\,\cdot\,\,;u_0,\Omega)$ is real-analytic on $\mathcal{B}_{\delta_0}(0)\subset C^{2,\alpha}(\Omega,u_0^*TN)$ in the sense of \cite[Section 3.13]{simon-book} and (i) follows. 
        \newline
        Now we turn to show (ii). Fix any $y_0\in N$ and define the sets 
        \begin{align*}
            U_{\delta_0}(y_0)&:=\{\xi\in T_{y_0}N \mbox{ : } \lvert\xi\rvert<\delta_0\}\\
            V_{\delta_0}(y_0)&:=\{\pi_N(y_0+\xi) \mbox{ : } \xi\in U_{\delta_0}(y_0)\}.
        \end{align*}
        Notice that, since $\pi_N$ is smooth, the map $\Phi_{y_0}:U_{\delta_0}(y_0)\to V_{\delta_0}(y_0)$ is a smooth diffeomorphism. Moreover, again by the smoothness of $\pi_N$, letting 
        \begin{align*}
            D:=\{(y_0,y)\in N\times N \mbox{ : } y\in V_{\delta_0}(y_0)\},
        \end{align*}
        the map $\Psi:D\to\R^k$ given by
        \begin{align*}
            \Psi(y_0,y):=\Phi_{y_0}^{-1}(y)\in U_{\delta_0}(y_0)\subset T_{y_0}N\subset\R^k
        \end{align*}
        is smooth on $D\subset N\times N$. Assume now that $u\in B_{\delta_0}(u_0)\subset C^2(\Omega,N)$. Then, by definition of $\Psi$, we have 
        \begin{align*}
            u=\pi_N(u_0+\Psi(u_0,u))
        \end{align*}
        and thus, setting $\tilde u:=\Psi(u_0,u)\in C^{2,\alpha}(\Omega,u_0^*TN)$, we get
        \begin{align*}
            u=\pi_N(u_0+\tilde u).
        \end{align*}
        Then, 
        \begin{align*}
            \mathscr{D}_M(u\,;\Omega)=\mathscr{D}_M(\pi_N(u_0+\tilde u)\,;\Omega)
        \end{align*}
        and (ii) follows immediately.
    \end{proof}
\end{lemma}
\begin{lemma}[First and second variation of $\mathscr{F}_{M}(\,\cdot\,,\Omega)$]\label{Lemma: the first and second variation of the energy discrepancy at the origin}
    Let $N\subset\R^k$ be a closed real-analytic submanifold in $\R^k$ and let $\pi_N:W_{\delta_0}(N)\to N$ be the nearest-point projection from a tubular neighbourhood $W_{\delta_0}(N)$ of $N$ into $N$. Let $(M^n,g)$ be a smooth $n$-dimensional Riemannian manifold and let $\Omega\subset M$ be an open subset of $M$. Let $u_0\in C^{\infty}(\Omega,N)$ be a smooth harmonic map on $\Omega$. Then, the zero section $0$ of the vector bundle $u_0^*TN$ is a critical point for $\mathscr{F}_{M}(\,\cdot\,\,;u_0,\Omega)$, i.e. 
    \begin{align*}
        \nabla\mathscr{F}_{M}(0\,;u_0,\Omega)[\varphi]=\nabla\mathscr{D}_{M}(u_0\,;\Omega)[\varphi]=0 \qquad\forall\,\varphi\in C^{2}(u_0^*TN).
    \end{align*}
    Moreover, the second variation of $\mathscr{F}_{M}(\,\cdot\,\,;u_0,\Omega)$ at the critical point $0$ is given by
    \begin{align*}
        \nabla^2\mathscr{F}_{M}(0\,;u_0,\Omega)[\varphi,\psi]&=\nabla^2\mathscr{D}_{M}(u_0\,;\Omega)[\varphi,\psi]\\
        &=\int_{\Omega}(\langle d\varphi,d\psi\rangle_g-\langle A_{u_0}(du_0,du_0)_g,A_{u_0}(\varphi,\psi)\rangle_g\, d\vol_{M},
    \end{align*}
    for every $\varphi,\psi\in C^{2}(u_0^*TN)$.
    \begin{proof}
        By definition, $u_0$ is a critical point of $\mathscr{D}_{M}(\,\cdot\,\,;\Omega)$, hence
        \begin{align*}
            \nabla\mathscr{D}_{M}(u_0\,;\Omega)[\varphi]=0 \qquad\forall\,\varphi\in C^{2}(u_0^*TN).
        \end{align*}
        Moreover, by Lemma \ref{Lemma: the analyticity of energy discrepancy in a neighbourhood of the origin} we know that $\mathscr{F}_{M}(\,\cdot\,\,;u_0,\Omega)$ is analytic in a neighbourhood of $0$ and, in particular, twice continuously Fr\'echet differentiable at $0$. Hence, the first variation is by definition given by
        \begin{align*}
            \nabla\mathscr{F}_M(0\,;u_0,\Omega)[\varphi]&=\frac{d}{dt}\mathscr{F}_M(t\varphi\,;u_0,\Omega)\bigg\rvert_{t=0}\\
            &=\frac{d}{dt}\mathscr{D}_M(\pi_N(u_0+t\varphi)\,;\Omega)\bigg\rvert_{t=0}=\nabla\mathscr{D}_M(u_0\,;\Omega)[\varphi]=0
        \end{align*}
        for every $\varphi\in C^{2,\alpha}(\Omega,u_0^*TN)$. This implies that $0$ is a critical point for $\mathscr{F}_{M}(\,\cdot\,\,;u_0,\Omega)$. Concerning the second variation of $\mathscr{F}_{M}(\,\cdot\,\,;u_0,\Omega)$ at $0$, again we have 
        \begin{align*}
            \nabla^2\mathscr{F}_M(0\,;u_0,\Omega)[\varphi,\psi]&=\nabla^2\mathscr{D}_M(u_0\,;\Omega)[\varphi,\psi]\\
            &=\int_{\Omega}(\langle d\varphi,d\psi\rangle_g-\langle A_{u_0}(du_0,du_0)_g,A_{u_0}(\varphi,\psi)\rangle_g\, d\vol_{M}
        \end{align*}
        for every $\varphi,\psi\in C^{2,\alpha}(\Omega,u_0^*TN)$, where the last equality follows by the standard second variation formula for the Dirichlet energy that is standard in literature and can be found for example in \cite{smith}. The statement follows. 
    \end{proof}
\end{lemma}
% \note{Above and below we have $A_{u_0}(du_0, du_0)_g$, do we need the subscript $g$?}
%
\begin{remark}\label{Remark: ellipticity of the second variation of the Dirichlet energy discrepancy}
    Under the same assumptions of Lemma \ref{Lemma: the analyticity of energy discrepancy in a neighbourhood of the origin} we point out that, by simple integration by parts, if $M$ is a closed smooth manifold we have that
    \begin{align}\label{Equation: second variation rewritten 1}
        \nabla^2\mathscr{F}_M(0\,;u_0)[\varphi,\psi]&=\int_{M}(\langle d\varphi,d\psi\rangle_g-\langle A_{u_0}(du_0,du_0)_g,A_{u_0}(\varphi,\psi)\rangle_g\, d\vol_{M}\\
        &=-\int_{M}(\langle \Delta_g\varphi,\psi\rangle_g+\langle A_{u_0}(du_0,du_0)_g,A_{u_0}(\varphi,\psi)\rangle_g\, d\vol_{M}.
    \end{align}
    Notice that the map
    \begin{align*}
        u_0^*TN\times u_0^*TN\ni(X,Y)\mapsto\langle A_{u_0}(du_0, du_0)_g,A_{u_0}(X,Y)\rangle_g
    \end{align*}
    is a symmetric smooth section over $M$ of the vector bundle $u_0^*T^*N\otimes u_0^*T^*N$. In particular, there exists a smooth section over $M$ of the vector bundle $u_0^*T^*N\otimes u_0^*TN$ such that
    \begin{align*}
        \langle A_{u_0}(du_0, du_0)_g,A_{u_0}(X,Y)\rangle_g=\langle S_{u_0}(du_0)_gX,Y\rangle_g \qquad\forall\, (X,Y)\in u_0^*TN\times u_0^*TN.
    \end{align*}
    Thus, by plugging the former definition in \eqref{Equation: second variation rewritten 1} we have
    \begin{align}\label{Equation: second variation rewritten 2} 
        \nabla^2\mathscr{F}_M(0\,;u_0)[\varphi,\psi]&=-\int_{M}(\langle\Delta_g\varphi+S_{u_0}(du_0)_g\varphi,\psi\rangle_g\, d\vol_{M} \qquad\forall\,\varphi,\psi\in C^2(u_0^*TN).
    \end{align}
    By introducing the linear elliptic operator $\mathcal{L}_{\mathscr{F}}:C^2(u_0^*TN)\to C^0(u_0^*TN)$ given by
    \begin{align*}
        \mathcal{L}_{\mathscr{F}}\varphi:=-(\Delta_g\varphi+S_{u_0}(du_0)_g\varphi) \qquad\forall\,\varphi\in C^2(u_0^*TN)
    \end{align*}
    we can rewrite \eqref{Equation: second variation rewritten 2} as
    \begin{align}\label{Equation: second variation rewritten 3} 
        \nabla^2\mathscr{F}_M(0\,;u_0)[\varphi,\psi]&=\int_{M}\langle\mathcal{L}_{\mathscr{F}}\varphi,\psi\rangle_g\, d\vol_{M} \qquad\forall\,\varphi,\psi\in C^2(u_0^*TN).
    \end{align}
    Notice that $\mathcal{L}_{\mathscr{F}}$ has compact resolvent, because it is an elliptic operator on the compact manifold $M$. Moreover, again by standard elliptic theory (i.e. Schauder estimates), for every $\alpha\in (0,1)$ we have that $\mathcal{L}_{\mathscr{F}}:C^{2,\alpha}(u_0^*TN)\to C^{0,\alpha}(u_0^*TN)$ has closed range. 
    This facts will play a crucial role in the following sections. 
\end{remark}
\section{A slicing lemma and Lyapunov--Schmidt reduction} \label{sec: slicing and LS reduction}
In this section we write the energy discrepancy introduced in Definition \ref{def: discrepancy energy} in a more convenient form. Besides, we recall the classical Lyapunov-Schmidt reduction and adapt its statement to our setting. 
\begin{lemma}[Slicing lemma]\label{Lemma: slicing lemma}
    Let $N\subset\R^k$ be a closed smooth submanifold in $\R^k$ and let $n\in\N$ be such that $n\ge 3$. Let $u_0\in W^{1,2}(\mathbb{S}^{n-1},N)$ and denote by $\tilde u_0\in W^{1,2}(\mathbb{B}^n,N)$ the 0-homogeneous extension of $u_0$ inside $\mathbb{B}^n$, i.e.
    \begin{align*}
        \tilde u_0:=\bigg(\frac{\cdot}{\lvert\,\cdot\,\rvert}\bigg)^*u_0.
    \end{align*}
    Then, for every $u\in W^{1,2}(\mathbb{B}^n,N)$ we have
    \begin{align*}
        \mathscr{E}_{\mathbb{B}^n}(u\,;\tilde u_0)=\int_0^1\mathscr{E}_{\mathbb{S}^{n-1}}(u(\rho,\,\cdot\,)\,;u_0)\rho^{n-3}\, d\mathcal{L}^1(\rho)+\int_0^1\int_{\mathbb{S}^{n-1}}\lvert\partial_{\rho}u(\rho,\theta)\rvert^2\, d\mathscr{H}^{n-1}(\theta)\rho^{n-1}\, d\mathcal{L}^1(\rho),
    \end{align*}
    where $(\rho,\theta)$ are polar coordinates centered at the origin of $\R^n$. Moreover, in case $u$ does not depend on $\rho$, the above simplifies to 
    \begin{equation*}
        \mathscr{E}_{\mathbb{B}^n}(u\,;\tilde u_0)= \frac{1}{n - 2}\mathscr{E}_{\mathbb{S}^{n-1}}(u(1,\cdot)\,;u_0). 
    \end{equation*}
    \begin{proof}
        By definition of the Dirichlet energy discrepancy and passing in polar coordinates $(\rho,\theta)$ centered at the origin, we get
        \begin{align*}
            \mathscr{E}_{\mathbb{B}^n}(u\,;\tilde u_0)&=\mathscr{D}_{\mathbb{B}^n}(u)-\mathscr{D}_{\mathbb{B}^n}(\tilde u_0)=\int_{\mathbb{B}^n}\lvert du\rvert^2\, d\mathcal{L}^n-\int_{\mathbb{B}^n}\lvert d\tilde u_0\rvert^2\, d\mathcal{L}^n\\
            &=\int_0^1\int_{\mathbb{S}^{n-1}}\bigg(\lvert\partial_{\rho}u(\rho,\theta)\rvert^2+\frac{1}{\rho^2}\lvert\partial_{\theta}u(\rho,\theta)\rvert^2\bigg)\, d\mathscr{H}^{n-1}(\theta)\rho^{n-1}\,d\mathcal{L}^1(\rho)\\
            &\quad-\int_0^1\int_{\mathbb{S}^{n-1}}\bigg(\lvert\partial_{\rho}\tilde u_0(\rho,\theta)\rvert^2+\frac{1}{\rho^2}\lvert\partial_{\theta}\tilde u_0(\rho,\theta)\rvert^2\bigg)\, d\mathscr{H}^{n-1}(\theta)\rho^{n-1}\,d\mathcal{L}^1(\rho)\\
            &=\int_0^1\int_{\mathbb{S}^{n-1}}\bigg(\lvert\partial_{\rho}u(\rho,\theta)\rvert^2+\frac{1}{\rho^2}\lvert\partial_{\theta}u(\rho,\theta)\rvert^2\bigg)\, d\mathscr{H}^{n-1}(\theta)\rho^{n-1}\,d\mathcal{L}^1(\rho)\\
            &\quad-\int_0^1\int_{\mathbb{S}^{n-1}}\frac{1}{\rho^2}\lvert\partial_{\theta}u_0(\theta)\rvert^2\, d\mathscr{H}^{n-1}(\theta)\rho^{n-1}\,d\mathcal{L}^1(\rho) \\
            %\\
            %&=\int_0^1\int_{\mathbb{S}^{n-1}}\lvert\partial_{\rho}u(\rho,\theta)\rvert^2\, d\mathscr{H}^{n-1}(\theta)\rho^{n-1}\,d\mathcal{L}^1(\rho)\\
            %&\quad+\int_0^1\int_{\mathbb{S}^{n-1}}\lvert\partial_{\theta}u(\rho,\theta)\rvert^2\bigg)\, d\mathscr{H}^{n-1}(\theta)\rho^{n-3}\,d\mathcal{L}^1(\rho)\\
            %&\quad-\int_0^1\int_{\mathbb{S}^{n-1}}\lvert\partial_{\theta}u_0(\theta)\rvert^2\, d\mathscr{H}^{n-1}(\theta)\rho^{n-3}\,d\mathcal{L}^1(\rho)\\
            &=\int_0^1\int_{\mathbb{S}^{n-1}}\lvert\partial_{\rho}u(\rho,\theta)\rvert^2\, d\mathscr{H}^{n-1}(\theta)\rho^{n-1}\,d\mathcal{L}^1(\rho)\\
            &\quad+\int_0^1\bigg(\int_{\mathbb{S}^{n-1}}\lvert\partial_{\theta}u(\rho,\,\cdot\,)\rvert^2\, d\mathscr{H}^{n-1}-\int_{\mathbb{S}^{n-1}}\lvert\partial_{\theta}u_0\rvert^2\, d\mathscr{H}^{n-1}\bigg)\rho^{n-3}\,d\mathcal{L}^1(\rho)\\
            &=\int_0^1\int_{\mathbb{S}^{n-1}}\lvert\partial_{\rho}u(\rho,\theta)\rvert^2\, d\mathscr{H}^{n-1}(\theta)\rho^{n-1}\,d\mathcal{L}^1(\rho)\\
            &\quad+\int_0^1\mathscr{E}_{\mathbb{S}^{n-1}}(u(\rho,\,\cdot\,)\,;u_0)\rho^{n-3}\,d\mathcal{L}^1(\rho).
        \end{align*}
        The statement follows. 
    \end{proof}
\end{lemma}
The next lemma is simply a particular case of the standard Lyapunov--Schmidt reduction for analytic functionals, adapted to our setting. For a proof of completely analogous statements, see e.g. \cite[Lemma B.1]{EngelsteinSpolaorVelichkov} or \cite[Lemma 2.2. and Appendix A]{engelstein-neumayer-spolaor}.
\begin{lemma}[Lyapunov--Schmidt reduction]\label{Lyapunov-Schmidt reduction}
    Let $N\subset\R^k$ be a closed real-analytic submanifold in $\R^k$ and let $n\ge 3$. Let $u_0\in C^{\infty}(\mathbb{S}^{n-1},N)$ be a smooth harmonic map on $\mathbb{S}^{n-1}$. Let 
    \begin{align*}
        K:=\ker\nabla^2\mathscr{F}_{\mathbb{S}^{n-1}}(0\,;u_0)\subset C^{\infty}(u_0^*TN)
    \end{align*}
    and let $K^{\perp}$ be its orthogonal complement in $L^2(u_0^*TN)$\footnote{Here and throughout, by $L^2(u_0^*TN)$ we mean the space of the $L^2$ sections of the smooth vector bundle $u_0^*TN$.}. Denote by $P_K$ and $P_{K^{\perp}}$ the $L^2$-orthogonal linear projection operators on the subspaces $K$ and $K^{\perp}$ respectively. There exist an open neighbourhood $U\subset K$ of $0$ in $K$ and an analytic function $F:U\to K^{\perp}$ such that the following facts hold.
    \begin{enumerate}[(i)]
        \item $F(0)=0$ and $\nabla F(0)=0$.
        \item $P_{K^{\perp}}(\nabla\mathscr{F}_{\mathbb{S}^{n-1}}(\varphi+F(\varphi)))=0$ for every $\varphi\in U$.
        \item $P_{K}(\nabla\mathscr{F}_{\mathbb{S}^{n-1}}(\varphi+F(\varphi)))=\nabla\mathfrak{q}(\varphi)$ for every $\varphi\in U$, where $\mathfrak{q}:U\to\R$ is the analytic map on $U$ given by 
        \begin{align*}
            \mathfrak{q}(\varphi):=\varphi+F(\varphi) \quad\forall\,\varphi\in U.
        \end{align*}
        \item There exists a constant $C>0$ such that very $\varphi,\eta\in U$, we have 
        \begin{align*}
            \|\nabla F(\varphi)[\eta]\|_{C^{2,\alpha}(u_0^*TN)}\le C\|\eta\|_{C^{0,\alpha}(u_0^*TN)}.
        \end{align*}
    \end{enumerate}
\end{lemma}
\section{The symmetric log-epiperimetric inequality for harmonic maps (Theorem \ref{Theorem: (log)-epiperimetric inequality for harmonic maps})}  \label{sec: proof of epiperimetric ineq}
As by the assumptions of Theorem \ref{Theorem: (log)-epiperimetric inequality for harmonic maps}, let $N\subset\R^k$ be a closed real-analytic submanifold in $\R^k$ and let $n\in\N$ be such that $n\ge 3$. Let $u_0\in C^{\infty}(\mathbb{S}^{n-1},N)$ be a harmonic map on $\mathbb{S}^{n-1}$. By Remark \ref{Remark: ellipticity of the second variation of the Dirichlet energy discrepancy}, we know that 
\begin{align*}
    K:=\ker\nabla^2\mathscr{F}_{\mathbb{S}^{n-1}}(0\,;u_0)
\end{align*}
is a finite-dimensional linear subspace of $C^{\infty}(u_0^*TN)$. Let $K^{\perp}\subset L^2(u_0^*TN)$ be the orthogonal complement of $K$ inside $L^2(u_0^*TN)$. Let $0\in U\subset K$ and  $F:U\to K^{\perp}$ be the real-analytic function given by the Lyapunov--Schmidt reduction (Lemma \ref{Lyapunov-Schmidt reduction}) of $\mathscr{F}_{\mathbb{S}^{n-1}}(\,\cdot\,\,;u_0)$ at the critical point $0$.  Denote by $P_K$ and $P_{K^{\perp}}$ the $L^2$-orthogonal linear projection operators on the subspaces $K$ and $K^{\perp}$ respectively.
\newline
Let $\pi_N:W_{\delta_0}(N)\to N$ be the nearest-point projection from a tubular neighbourhood $W_{\delta_0}(N)$ of $N$ into $N$. Fix any $\alpha\in(0,1)$, and let $\delta\in (0,\delta_0)$ be small enough so that for every $u\in C^{2,\alpha}(\mathbb{S}^{n-1},N)$ such that 
\begin{align*}
    \|u-u_0\|_{C^{2,\alpha}(\mathbb{S}^{n-1})}<\delta
\end{align*}
we have $P_K\varphi_u\in U$ and $\varphi_u,F(P_K\varphi_u)\in\mathcal{B}_{\delta_0/2}(0)$, where $\varphi_u\in C^{2,\alpha}(u_0^*TN)$ is such that $u=\pi_N(u_0+\varphi_u)$, and where we denoted by $\mathcal{B}_r(f)$ the ball of radius $r$ in $C^{2,\alpha}(u_0^*TN)$. 
Notice that $\delta\in (0,\delta_0)$ can always be chosen in such a way in view of Lemma \ref{Lemma: the analyticity of energy discrepancy in a neighbourhood of the origin}-(ii) and thanks to the properties of the Lyapunov-Schmidt reduction (Lemma \ref{Lyapunov-Schmidt reduction}-(i),(iv)).
Then, we write
\begin{align*}
    \varphi_u&=P_K\varphi_u+P_{K^{\perp}}\varphi_u\\
    & = P_K\varphi_u+F(P_K\varphi_u)+(P_{K^{\perp}}\varphi_u-F(P_K\varphi_u))\\
    &=P_K\varphi_u+F(P_K\varphi_u)+\varphi_u^{\perp},
\end{align*}
where we have defined 
\begin{align*}
    \varphi_u^{\perp}:=P_{K^{\perp}}\varphi_u-F(P_K\varphi_u)\in K^{\perp}.
\end{align*}
Since the second variation $\nabla^2\mathscr{F}_{\mathbb{B}^n}(0\,;u_0)$ is induced by an elliptic operator $\mathcal{L}_{\mathscr{F}}$ on a compact manifold (see Remark \ref{Remark: ellipticity of the second variation of the Dirichlet energy discrepancy}) and since every elliptic operator on a compact manifold has compact resolvent, by the spectral theory for operators with compact resolvent we know that there exist a countable orthonormal basis $\{\phi_j\}_{j\in\N}\subset C^{\infty}(u_0^*TN)$ of $W^{1,2}(u_0^*TN)$ and countably many real numbers\footnote{This follows from the symmetry of $\nabla^2\mathscr{F}_{\mathbb{B}^n}(0\,;u_0)$, which translates in the self-adjointness of $\mathcal{L}_{\mathscr{F}}$.} $\{\lambda_j\}_{j\in\N}$ such that 
\begin{align*}
    \mathcal{L}_{\mathscr{F}}\phi_j=\lambda_j\phi_j \qquad\forall\,j\in\N. 
\end{align*}
Moreover, every eigenvalue $\lambda_j$ of $\mathcal{L}_{\mathscr{F}}$ has finite multiplicity. We let
\begin{align*}
    \ell:=\dim K<+\infty
\end{align*}
and we assume that the eigenfunctions $\phi_j$ are ordered in such a way that $\{\phi_1,...,\phi_{\ell}\}$ form an orthonormal basis of $K$. Define the index sets 
\begin{align*}
    J_{+}&:=\{j\in\N \mbox{ : } \lambda_j>0\}\\
    J_{-}&:=\{j\in\N \mbox{ : } \lambda_j<0\}
\end{align*}
and we $\{a_j\}_{j\in J_{-}\cap J_{+}}\subset\R$ and $\{b_{1},...,b_{\ell}\}\subset\R$ be such that 
\begin{align*}
    &\varphi_u^{\perp}=\sum_{j\in J_{-}}a_j\phi_j+\sum_{j\in J_{+}}a_j\phi_j=:\varphi_{u,-}^{\perp}+\varphi_{u,+}^{\perp}, \qquad \text{and} \qquad P_K\varphi_u=\sum_{j=1}^{\ell}b_j\phi_j.
\end{align*}
Since $\varphi_u,F(P_K\varphi_u)\in\mathcal{B}_{\frac{\delta_0}{2}}(0)$ and $P_K\varphi_u\in U$, there exists $\xi>0$ with $B_{\xi}^{\ell}(b)\subset\R^{\ell}$ such that for every $x=(x_1,...,x_{\ell})\in B_{\xi}^{\ell}(b)$ we have 
\begin{align*}
\sum_{j=1}^{\ell}x_j\phi_j\in U \qquad \text{and} \qquad \sum_{j=1}^{\ell}x_j\phi_j+F\bigg(\sum_{j=1}^{\ell}x_j\phi_j\bigg)\in\mathcal{B}_{\frac{\delta_0}{2}}(0).
\end{align*}
Let $f:B_{\xi}^{\ell}(b)\subset\R^{\ell}\to\R$ be the real-analytic function given by 
\begin{align} \label{def: f}
    f(x):=\mathscr{F}_{\mathbb{S}^{n-1}}\bigg(\sum_{j=1}^{\ell}x_j\phi_j+F\bigg(\sum_{j=1}^{\ell}x_j\phi_j\bigg)\,;u_0\bigg) \qquad\forall\,x\in B_{\xi}^{\ell}(b).
\end{align}
Let $t_0\in(0,1)$ and let $v:[0,t_0]\to B_{\xi}^{\ell}(b)$ the smooth vector field on $B_{\xi}^{\ell}(b)$ solving on $[0,t_0]$ the following normalized gradient flow equation for $f$ with initial condition $b=(b_1,...,b_{\ell})\in\R^{\ell}$:
\begin{align*}
    v'(t)&=\begin{cases}\displaystyle{-\frac{\nabla f(v(t))}{\lvert\nabla f(v(t))\rvert}} & \mbox{ if } f(v(t))>\displaystyle{\frac{f(b)}{2}}\\0 & \mbox{ otherwise;}\end{cases}\\
    v(0)&=b.
\end{align*}
Note that this gradient flow is a \textit{finite dimensional harmonic map heat flow}. Let then $\eta,\eta_-,\eta_+:[0,1]\to\R$ be the cut-off functions given by  
\begin{align} \label{def: eta and eta+}
    \eta(\rho):=\varepsilon_ff(b)^{1-\gamma}\sqrt{n}C(1-\rho), \quad \eta_-(\rho):=1+(1-\rho)\beta\varepsilon \quad \text{and} \quad 
    \eta_+(\rho):=1-(1-\rho)\alpha\varepsilon,
\end{align}
for all $\rho \in [0, 1]$, and where $\varepsilon,\varepsilon_f,C,\alpha,\beta >0$ and $\gamma\in[0,1)$ are parameters to be chosen later in the proof. For now we just assume that 
\begin{align*}
    \varepsilon_ff(b)^{1-\gamma}\sqrt{n}C < t_0,
\end{align*}
so that $0\le\eta<t_0$. Then, let $\mu:[0,1]\to U$ be given by
\begin{align*}
    \mu(\rho):=\sum_{j=1}^{\ell}v_j(\eta(\rho))\phi_j(\theta) \qquad\forall\,\rho\in[0,1] \quad \forall \, \theta \in \mathbb{S}^{n - 1}. 
\end{align*}
Define $\varphi_{\hat u}\in W^{1,2}(\mathbb{B}^n,u_0^*TN)$ by
\begin{align}\label{Equation: definition of the competitor before projection}
    \varphi_{\hat u}(\rho,\,\cdot\,):=\mu(\rho)+F(\mu(\rho))+\eta_-(\rho)\varphi_{u,-}^{\perp}(\cdot) + \eta_+(\rho)\varphi_{u,+}^{\perp}(\cdot)
\end{align}
for every $\rho\in(0,1]$. Notice that $\varphi_{\hat u}(1,\,\cdot\,)=\varphi_u$ and that $\varphi_{\hat u}(\rho,\,\cdot\,)\in C^{2,\alpha}(\mathbb{S}^{n-1},N)$ for every $\rho\in(0,1]$. Moreover, for every $\rho\in(0,1]$ we have
\begin{align*}
    \|\varphi_{\hat u}(\rho,\,\cdot\,)-\varphi_u\|_{C^2(\mathbb{S}^{n-1})}&=\|\varphi_{\hat u}(\rho,\,\cdot\,)-\varphi_{\hat u}(1,\,\cdot\,)\|_{C^2(\mathbb{S}^{n-1})}\\
    &\le\tilde C\big(\|\eta'\|_{L^{\infty}((0,1))}+\|\eta_{-}'\|_{L^{\infty}((0,1))}+\|\eta_{+}'\|_{L^{\infty}((0,1))}\big)<\frac{\delta_0}{4}
\end{align*}
where $\tilde C>0$ depends only on the $C^{2}$-norms of the eigenfunctions $\phi_1,...,\phi_{\ell}$ and the constants $\varepsilon,\varepsilon_f,C,\alpha>0$, $\gamma\in[0,1)$ are properly chosen in the definition of the cut-off functions $\eta$ and $\eta_+$ in such a way that
\begin{align*}
    \varepsilon_ff(b)^{1-\gamma}\sqrt{n}C<\frac{\delta_0}{12\tilde C} \quad \text{and} \quad \alpha\varepsilon <\frac{\delta_0}{12\tilde C}.
\end{align*}
Thus we have 
\begin{align}\label{Equation: the slices are close to zero}
    \|\varphi_{\hat u}(\rho,\,\cdot\,)\|_{C^2(\mathbb{S}^{n-1})}\le\|\varphi_{\hat u}(\rho,\,\cdot\,)-\varphi_u\|_{C^2(\mathbb{S}^{n-1})}+\|\varphi_u\|_{C^2(\mathbb{S}^{n-1})}<\frac{\delta_0}{2} \qquad\forall\,\rho\in(0,1].
\end{align}
Hence, lastly we can define the competitor $\hat u\in W^{1,2}(\mathbb{B}^n,N)$ by
\begin{align*}
    \hat u(\rho,\,\cdot\,):=\pi_N(u_0+\varphi_{\hat u}(\rho,\,\cdot\,))
\end{align*}
for every $\rho\in(0,1]$. First, notice that
\begin{align*}
    \hat u|_{\mathbb{S}^{n-1}}=\pi_N(u_0+\varphi_{\hat u}(1,\,\cdot\,))=\pi_N(u_0+\varphi_{u})=u.
\end{align*}
By Lemma \ref{Lemma: slicing lemma} we have the estimate 
\begin{align} \label{equation: starting estimate}
\begin{split}
    \mathscr{E}_{\mathbb{B}^n}(\hat u\,;\tilde u_0)-\mathscr{E}_{\mathbb{B}^n}(\tilde u\,;\tilde u_0)&\le\int_0^1\big(\mathscr{E}_{\mathbb{S}^{n-1}}(\hat u(\rho,\,\cdot\,)\,; u_0)-\mathscr{E}_{\mathbb{S}^{n-1}}(u\,;u_0)\big)\rho^{n-3}\, d\mathcal{L}^1(\rho)\\
    &\quad+\int_0^1\int_{\mathbb{S}^{n-1}}\lvert\partial_{\rho}\hat u(\rho,\theta)\rvert^2\, d\mathscr{H}^{n-1}(\theta)\rho^{n-1}\, d\mathcal{L}^1(\rho)=\operatorname{I}+\operatorname{II},
\end{split}
\end{align}
where we have defined
\begin{align*}
    \operatorname{I}&:=\int_0^1\big(\mathscr{E}_{\mathbb{S}^{n-1}}(\hat u(\rho,\,\cdot\,)\,; u_0)-\mathscr{E}_{\mathbb{S}^{n-1}}(u\,;u_0)\big)\rho^{n-3}\, d\mathcal{L}^1(\rho)\\
    \operatorname{II}&:=\int_0^1\int_{\mathbb{S}^{n-1}}\lvert\partial_{\rho}\hat u(\rho,\theta)\rvert^2\, d\mathscr{H}^{n-1}(\theta)\rho^{n-1}\, d\mathcal{L}^1(\rho).
\end{align*}
Let $C_F>0$ be the constant given by Lemma \ref{Lyapunov-Schmidt reduction}-(iv) and notice that, by letting 
\begin{align*}
    C_N&:=\|d\pi_N\|_{L^{\infty}(\overline{W_{\frac{\delta_0}{2}}(N)})}^2<+\infty\\
    C_{n-1}&:=\mathscr{H}^{n-1}(\mathbb{S}^{n-1})\\
    \hat C&:=C_N(1+C_F^2)C_{n-1}\bigg(C^2+\frac{\alpha^2}{n}\bigg), 
\end{align*}
we have
\begin{align} \label{equation: estimate for II}
\begin{split}
    \operatorname{II}&=\int_0^1\int_{\mathbb{S}^{n-1}}\lvert\partial_{\rho}\hat u(\rho,\theta)\rvert^2\, d\mathscr{H}^{n-1}(\theta)\rho^{n-1}\, d\mathcal{L}^1(\rho)\\
    &\le C_N\int_0^1\int_{\mathbb{S}^{n-1}}\lvert\partial_{\rho}\varphi_{\hat u}(\rho,\theta)\rvert^2\, d\mathscr{H}^{n-1}(\theta)\rho^{n-1}\, d\mathcal{L}^1(\rho)\\
    &\le C_N\int_{0}^1\int_{\mathbb{S}^{n-1}}\big((\eta'(\rho))^2(1+\|\nabla F(v)[v']\|_{L^{\infty}((0,t_0))}^2)\, d\mathscr{H}^{n-1}(\theta)\, d\mathcal{L}^1(\rho)\\
    &\quad +C_N\int_{0}^1\int_{\mathbb{S}^{n-1}}(\eta_-'(\rho))^2\lvert\varphi_{u,-}^{\perp}(\theta)\rvert^2\big)\rho^{n-1}\, d\mathscr{H}^{n-1}(\theta)\, d\mathcal{L}^1(\rho)\\
    &\quad +C_N\int_{0}^1\int_{\mathbb{S}^{n-1}}(\eta_+'(\rho))^2\lvert\varphi_{u,+}^{\perp}(\theta)\rvert^2\big)\rho^{n-1}\, d\mathscr{H}^{n-1}(\theta)\, d\mathcal{L}^1(\rho)\\
    &\le C_N(1+C_F^2)C_{n-1}\int_0^1(\eta'(\rho))^2\rho^{n-1}\, d\mathcal{L}^1(\rho)+C_N\|\varphi_{u,-}^{\perp}\|_{L^2(\mathbb{S}^{n-1})}^2\int_0^1(\eta_{-}'(\rho))^2\rho^{n-1}\, d\mathcal{L}^1(\rho)\\
    &\quad +C_N\|\varphi_{u,+}^{\perp}\|_{L^2(\mathbb{S}^{n-1})}^2\int_0^1(\eta_{+}'(\rho))^2\rho^{n-1}\, d\mathcal{L}^1(\rho)\\
    &\le C_N(1+C_F^2)C_{n-1}\bigg(\int_0^1\varepsilon_f^2f(b)^{2-2\gamma}C^2n\rho^{n-1}\, d\mathcal{L}^1(\rho)+\|\varphi_{u,-}^{\perp}\|_{L^2(\mathbb{S}^{n-1})}^2\int_0^1\varepsilon^2\beta^2\rho^{n-1}\, d\mathcal{L}^1(\rho)\\
    &\quad +\|\varphi_{u,+}^{\perp}\|_{L^2(\mathbb{S}^{n-1})}^2\int_0^1\varepsilon^2\alpha^2\rho^{n-1}\, d\mathcal{L}^1(\rho)\bigg)\\
    &\le\hat C\Big(\varepsilon_f^2f(b)^{2-2\gamma}+\varepsilon^2\big(\|\varphi_{u,-}^{\perp}\|_{L^2(\mathbb{S}^{n-1})}^2+\|\varphi_{u,+}^{\perp}\|_{L^2(\mathbb{S}^{n-1})}^2\big)\Big)\\
    &\le\hat C\big(\varepsilon_f^2f(b)^{2-2\gamma}+\varepsilon^2\|\varphi_{u}^{\perp}\|_{W^{1,2}(\mathbb{S}^{n-1})}^2\big).
\end{split}
\end{align}
Now we turn to estimate $\operatorname{I}$. We notice that, by \eqref{Equation: the slices are close to zero}, we have
\begin{align*}
    \mathscr{E}_{\mathbb{S}^{n-1}}(\hat u(\rho,\,\cdot\,)\,; u_0)-\mathscr{E}_{\mathbb{S}^{n-1}}(u\,;u_0)&=\mathscr{F}_{\mathbb{S}^{n-1}}(\varphi_{\hat u}(\rho,\,\cdot\,)\,; u_0)-\mathscr{F}_{\mathbb{S}^{n-1}}(\varphi_u\,;u_0)\\
    &=\operatorname{III}+\operatorname{IV},
\end{align*}
with
\begin{align*}
    \operatorname{III}&:=\mathscr{F}_{\mathbb{S}^{n-1}}(\varphi_{\hat u}(\rho,\,\cdot\,)\,; u_0)-\mathscr{F}_{\mathbb{S}^{n-1}}(\mu(\rho)+F(\mu(\rho))\,;u_0)\\
    &\,\,\quad-\big(\mathscr{F}_{\mathbb{S}^{n-1}}(\varphi_u\,;u_0)-\mathscr{F}_{\mathbb{S}^{n-1}}(P_K\varphi_u+F(P_K\varphi_u)\,;u_0)\big)\\
    \operatorname{IV}&:=\mathscr{F}_{\mathbb{S}^{n-1}}(\mu(\rho)+F(\mu(\rho))\,;u_0)-\mathscr{F}_{\mathbb{S}^{n-1}}(P_K\varphi_u+F(P_K\varphi_u)\,;u_0).
\end{align*}
Letting $\psi_{\rho}:=\varphi_{\hat u}(\rho,\,\cdot\,)-\mu(\rho) - F(\mu(\rho))\in K^{\perp}$, by Taylor expanding $\mathscr{F}_{\mathbb{S}^{n-1}}(\,\cdot\,\,;u_0)$ we get that
\begin{align*}
    \operatorname{III}&=\nabla\mathscr{F}_{\mathbb{S}^{n-1}}(\mu(\rho)+F(\mu(\rho)\,;u_0)[\psi_{\rho}]+\nabla^2 \mathscr{F}_{\mathbb{S}^{n-1}}(\mu(\rho)+F(\mu(\rho)+s_1\psi_{\rho}\,;u_0)[\psi_{\rho},\psi_{\rho}]\\
    &\quad-\nabla\mathscr{F}_{\mathbb{S}^{n-1}}(P_{K}\varphi_u+F(P_K\varphi_u)\,;u_0)[\varphi_{u}^{\perp}]\\
    &\quad-\nabla^2 \mathscr{F}_{\mathbb{S}^{n-1}}(P_{K}\varphi_u+F(P_K\varphi_u)+s_2\varphi_u^{\perp}\,;u_0)[\varphi_u^{\perp},\varphi_u^{\perp}]\\
    &=\nabla^2\mathscr{F}_{\mathbb{S}^{n-1}}(\mu(\rho)+F(\mu(\rho)+s_1\psi_{\rho}\,;u_0)[\psi_{\rho},\psi_{\rho}]\\
    &\quad-\nabla^2\mathscr{F}_{\mathbb{S}^{n-1}}(P_{K}\varphi_u+F(P_K\varphi_u)+s_2\varphi_u^{\perp}\,;u_0)[\varphi_u^{\perp},\varphi_u^{\perp}]
\end{align*}
for some $s_1,s_2\in[0,1]$, where in the second equality we have used that $\psi_{\rho},\varphi_u^{\perp}\in K^{\perp}$ and Lemma \ref{Lyapunov-Schmidt reduction}-(ii). Note that to avoid having factors $1/2$ in front of the Hessian term, we will abuse notation and assume that our definition of Hessian already incorporates this value.  
By using the analyticity of $\mathscr{F}_{\mathbb{S}^{n-1}}(\,\cdot\,\,;u_0)$ around the zero section $0\in C^{2,\alpha}(u_0^*TN)$, and in particular the fact that its second variation is locally Lipschitz around $0\in C^{2,\alpha}(u_0^*TN)$ (it is actually smooth around $0\in C^{2,\alpha}(u_0^*TN)$) we get that there exists $L>0$ such that 
\begin{align*}
    \lvert\nabla^2\mathscr{F}_{\mathbb{S}^{n-1}}(\xi\,;u_0)[\zeta,\zeta]-\nabla^2\mathscr{F}_{\mathbb{S}^{n-1}}(0\,;u_0)[\zeta,\zeta]\rvert\le L\|\xi\|_{C^{2,\alpha}(\mathbb{S}^{n-1})}\|\zeta\|_{W^{1,2}(\mathbb{S}^{n-1})}^2
\end{align*}
for every $\zeta\in C^{2,\alpha}(u_0^*TN)$ and $\xi\in C^{2,\alpha}(u_0^*TN)$ sufficiently close to $0$ in the $C^{2,\alpha}$-norm. Thus, we get
\begin{align}\label{Equation: estimate 1}
\begin{split}
    \operatorname{III}&\le\nabla^2\mathscr{F}_{\mathbb{S}^{n-1}}(0\,;u_0)[\psi_{\rho},\psi_{\rho}]-\nabla^2\mathscr{F}_{\mathbb{S}^{n-1}}(0\,;u_0)[\varphi_u^{\perp},\varphi_u^{\perp}]\\
    &\quad+L\|\mu(\rho)+F(\mu(\rho))+s_1\psi_{\rho}\|_{C^{2,\alpha}(\mathbb{S}^{n-1})}\|\psi_{\rho}\|_{W^{1,2}(\mathbb{S}^{n-1})}^2\\
    &\quad+L\|P_K\varphi_u+F(P_K\varphi_u)+s_2\varphi_u^{\perp}\|_{C^{2,\alpha}(\mathbb{S}^{n-1})}\|\varphi_u^{\perp}\|_{W^{1,2}(\mathbb{S}^{n-1})}^2\\
    &\le\nabla^2\mathscr{F}_{\mathbb{S}^{n-1}}(0\,;u_0)[\psi_{\rho},\psi_{\rho}]-\nabla^2\mathscr{F}_{\mathbb{S}^{n-1}}(0\,;u_0)[\varphi_u^{\perp},\varphi_u^{\perp}]\\
    &\quad+L\|\mu(\rho)+F(\mu(\rho))+s_1\psi_{\rho}\|_{C^{2,\alpha}(\mathbb{S}^{n-1})}\|\psi_{\rho}\|_{W^{1,2}(\mathbb{S}^{n-1})}^2\\
    &\quad+L\big(\|P_K\varphi_u\|_{C^{2,\alpha}(\mathbb{S}^{n-1})}+\|\varphi_u^{\perp}\|_{C^{2,\alpha}(\mathbb{S}^{n-1})}\big)\|\varphi_u^{\perp}\|_{W^{1,2}(\mathbb{S}^{n-1})}^2
\end{split}
\end{align}
Notice that, by definition \eqref{Equation: definition of the competitor before projection}, we have
\begin{align*}
    \psi_{\rho}=\varphi_{\hat u}(\rho,\,\cdot\,)-\mu(\rho) - F(\mu(\rho))=\eta_-(\rho)\varphi_{u,-}^{\perp}+\eta_{+}(\rho)\varphi_{u,+}^{\perp}.
\end{align*}
Hence
\begin{align}\label{Equation: estimate 2}
\begin{split}
    \nabla^2&\mathscr{F}_{\mathbb{S}^{n-1}}(0\,;u_0)[\psi_{\rho},\psi_{\rho}]-\nabla^2\mathscr{F}_{\mathbb{S}^{n-1}}(0\,;u_0)[\varphi_u^{\perp},\varphi_u^{\perp}] \\
    &=(\eta_{-}(\rho)^2-1)\nabla^2\mathscr{F}_{\mathbb{S}^{n-1}}(0\,;u_0)[\varphi_{u,-}^{\perp},\varphi_{u,-}^{\perp}] +(\eta_{+}(\rho)^2-1)\nabla^2\mathscr{F}_{\mathbb{S}^{n-1}}(0\,;u_0)[\varphi_{u,+}^{\perp},\varphi_{u,+}^{\perp}]
\end{split}
\end{align}
and
\begin{align}\label{Equation: estimate 3}
\begin{split}
    \|\mu(\rho)+F(\mu(\rho))&+s_1\psi_{\rho}\|_{C^{2,\alpha}(\mathbb{S}^{n-1})}\|\psi_{\rho}\|_{W^{1,2}(\mathbb{S}^{n-1})}^2\\
    &\le C\big(\|\mu(\rho)\|_{C^{2,\alpha}(\mathbb{S}^{n-1})}+\|\varphi_u^{\perp}\|_{C^{2,\alpha}(\mathbb{S}^{n-1})}\big)\|\varphi_u^{\perp}\|_{W^{1,2}(\mathbb{S}^{n-1})}^2.
\end{split}
\end{align}
By plugging \eqref{Equation: estimate 2} and \eqref{Equation: estimate 3} in \eqref{Equation: estimate 1} we get
\begin{align} \label{equation: estimate III before integrating}
\begin{split}
    \operatorname{III}&\le(\eta_{-}(\rho)^2-1)\nabla^2\mathscr{F}_{\mathbb{S}^{n-1}}(0\,;u_0)[\varphi_{u,-}^{\perp},\varphi_{u,-}^{\perp}] +(\eta_{+}(\rho)^2-1)\nabla^2\mathscr{F}_{\mathbb{S}^{n-1}}(0\,;u_0)[\varphi_{u,+}^{\perp},\varphi_{u,+}^{\perp}]\\
    &\quad +C\big(\|\mu(\rho)\|_{C^{2,\alpha}(\mathbb{S}^{n-1})}+\|P_K\varphi_u\|_{C^{2,\alpha}(\mathbb{S}^{n-1})}+2\|\varphi_u^{\perp}\|_{C^{2,\alpha}(\mathbb{S}^{n-1})}\big)\|\varphi_u^{\perp}\|_{W^{1,2}(\mathbb{S}^{n-1})}^2.
\end{split}
\end{align}
By definition of $\eta_+$ and $\eta_-$, cf. \eqref{def: eta and eta+}, up to choosing $\alpha,\beta>0$ big enough and $\eps>0$ small enough depending on $n\ge 3$ we have
\begin{equation*}
  \int_{0}^{1}(\eta_{-}^{2}(\rho) - 1) \rho^{n - 3} \, d\rho \geq \frac{4}{n-2}\,\varepsilon.   
\end{equation*}
and
\begin{equation*}
  \int_{0}^{1}(\eta_{+}^{2}(\rho) - 1) \rho^{n - 3} \, d\rho \leq -\frac{4}{n-2}\,\varepsilon.   
\end{equation*}
Consequently, multiplying \eqref{equation: estimate III before integrating} by $\rho^{n - 3}$ and integrating it with respect to $\rho$, we infer
\begin{align} \label{equation: estimate III after integrating}
\begin{split}
    \int_{0}^{1} & \operatorname{III} \rho^{n - 3} \, d\mathcal{L}^1(\rho) \\
    & \leq\frac{4}{n-2}\varepsilon \max_{\lambda_j < 0} \lambda_j \Vert \varphi_{u,-}^{\perp} \Vert_{W^{1, 2}(\mathbb{S}^{n - 1})}^2 - \frac{4}{n-2}\,\varepsilon \min_{\lambda_j > 0} \lambda_j \Vert \varphi_{u,+}^{\perp} \Vert_{W^{1, 2}(\mathbb{S}^{n - 1})}^2 \\
    & \quad + C\big(\|\mu(\rho)\|_{C^{2,\alpha}(\mathbb{S}^{n-1})}+\|P_K\varphi_u\|_{C^{2,\alpha}(\mathbb{S}^{n-1})}+\|\varphi_u^{\perp}\|_{C^{2,\alpha}(\mathbb{S}^{n-1})}\big)\|\varphi_u^{\perp}\|_{W^{1,2}(\mathbb{S}^{n-1})}^2 \\
    & \leq - \left( C_{u_0} \varepsilon - C\big(\|\mu(\rho)\|_{C^{2,\alpha}(\mathbb{S}^{n-1})}+\|P_K\varphi_u\|_{C^{2,\alpha}(\mathbb{S}^{n-1})}+\|\varphi_u^{\perp}\|_{C^{2,\alpha}(\mathbb{S}^{n-1})}\big) \right) \|\varphi_u^{\perp}\|_{W^{1,2}(\mathbb{S}^{n-1})}^2, 
\end{split} 
\end{align}
where $C_{u_0} > 0$ is a constant depending only on $n$ and on the spectral gap of the second variation of $\mathscr{F}(\,\cdot\,;\,u_0)$ at $0$, given by
\begin{align*}
    C_{u_0}:=\frac{4}{n-2}\Big(\min_{\lambda_j > 0} \lambda_j -\max_{\lambda_j < 0} \lambda_j\Big).
\end{align*}
We remark that here we need $n \geq 3$ to have integrability of the term $C \int_{0}^{1} \rho^{n - 3} d\rho$. 
Notice now that 
\begin{equation*}
    \vert \mu(\rho) - P_K(\varphi_u) \vert \leq \int_{0}^{\eta(\rho)} \vert \mu^\prime(t) \vert \; dt \leq \vert \eta(\rho) \vert \leq C\varepsilon_f f(b)^{\gamma}, 
\end{equation*}
as well as 
\begin{equation*}
    \left\vert \frac{d}{d\rho} \mu(\rho) \right\vert \leq \varepsilon_f f(b)^{\gamma}, 
\end{equation*}
so that choosing $\varepsilon_f$ sufficiently small, and combining these estimates with elliptic regularity, we have the estimate 
\begin{equation*}
    \Vert \mu(\rho) \Vert_{C^{2, \alpha}(\mathbb{S}^{n - 1})} \leq 2 \Vert P_K\varphi_u \Vert_{C^{2, \alpha}(\mathbb{S}^{n - 1})}. 
\end{equation*}
Whence, choosing $\delta>0$ sufficiently small (depending on $C_{u_0}$) and plugging $\|\varphi_u\|_{C^{2, \alpha}}(\mathbb{S}^{n - 1})<\delta$ in \eqref{equation: estimate III after integrating} , we infer 
\begin{equation} \label{equation: final integral estimate for III}
    \int_{0}^{1} \operatorname{III} \rho^{n - 3} \; d\rho \leq - C_{u_0} \varepsilon \|\varphi_u^{\perp}\|_{W^{1,2}(\mathbb{S}^{n-1})}^2. 
\end{equation}
We are now left with estimating $\operatorname{IV}$. To this sake, we record \L ojasiewicz's inequality for analytic function in $\mathbb{R}^l$, cf. \cite{Loj}. 
\begin{lemma} \label{lemma: finite dimensional Lojasiewicz inequality}
    Consider an open set $U \subset \mathbb{R}^l$, and an analytic function $h \colon U \rightarrow \mathbb{R}$. For every critical point $x \in U$ of $h$, there exist a neighborhood $V$ of $x$, an exponent $\gamma \in (0, 1/2]$, and a constant $K\ge 2$ such that 
    \begin{equation*}
        \vert h(x) - h(y) \vert^{1 - \gamma} \leq K \vert \nabla h(y) \vert, 
    \end{equation*}
    for all $y \in V$. 
\end{lemma}
In particular, we can apply Lemma \ref{lemma: finite dimensional Lojasiewicz inequality} to $f$ defined in \eqref{def: f}, and infer the existence of a neighborhood $V$ of the origin, constants $K>0$ and $\gamma \in (0, 1/2]$ depending on $u_0$ and $n$, such that $\vert f(v) \vert^{1 - \gamma} \leq K\vert \nabla f(v) \vert$, for every $v \in V$. 
Then, if $f(v(s)) > 0$, for $0 < s < t$, we have 
\begin{equation} \label{equation: FTC}
    f(v(t)) - f(v(0)) = f(v(t)) - f(b) = \int_{0}^{t} \nabla f(v(\tau)) \cdot v^\prime(\tau) \; d\tau = - \int_{0}^{t} \vert \nabla f(v(\tau)) \vert \; d\tau \leq 0. 
\end{equation}
This implies that the function $t \mapsto f(v(t))$ is non-increasing, which in turn implies the existence of $\overline{\tau} > 0$ such that $f(v(t)) \geq f(b)/2 > 0$, for $0 \leq t \leq \overline{\tau}$, and $f(v(t)) \leq f(b)/2$ if $t \geq \overline{\tau}$. We now have two cases. If $\eta(\rho) \leq \overline{\tau}$, we have the following 
\begin{alignat*}{2}
    \operatorname{IV} & = f(v(\eta(\rho))) - f(b) \\
    & \leq - \int_{0}^{\eta(\rho)} \vert \nabla f(v(\tau)) \vert \, d\tau  \qquad && \text{from \eqref{equation: FTC}} \\
    & \leq - K\int_{0}^{\eta(\rho)} \vert f(v(\tau)) \vert^{1 - \gamma} \; d\tau \qquad && \text{from Lemma \ref{lemma: finite dimensional Lojasiewicz inequality}} \\
    & \leq - Kf(v(\eta(\rho)))^{1 - \gamma} \eta(\rho) \qquad && \text{monotonicity of $f$} \\
    & \leq - \frac{K}{2^{1 - \gamma}} f(b)^{1 - \gamma} \eta(\rho)\qquad && \text{definition of $\overline{\tau}$}\\
    & \leq - \frac{K}{2} f(b)^{1 - \gamma} \eta(\rho).
\end{alignat*}
Otherwise, if $\eta(\rho) > \overline{\tau}$, we have
\begin{align*}
\operatorname{IV} = f(v(\eta(\rho))) - f(b) < - \frac{1}{2}f(b) < - \eta(\rho)f(b)^{1 - \gamma}, 
\end{align*}
where for the last inequality we used the inequality $\vert \eta \vert \leq C \varepsilon_f f(b)^{1 - \gamma} < 1/2$ which holds as long as $f(b)$ is small enough. Thus, we obtain
\begin{align*}
    \operatorname{IV}\le - \eta(\rho)f(b)^{1 - \gamma}, 
\end{align*}
and this concludes the estimate for $\operatorname{IV}$. 

We are now able to finish the proof of Theorem \ref{Theorem: (log)-epiperimetric inequality for harmonic maps}. First notice that, by the slicing Lemma \ref{Lemma: slicing lemma}, we have
\begin{align*}
    \mathscr{E}_{\mathbb{B}^n}(\tilde u\,;\tilde u_0) & = \frac{1}{n - 2}\mathscr{E}_{\mathbb{S}^{n-1}}(u \; ;u_0) \\
    & = \frac{1}{n - 2}\Big(\big(\mathscr{F}_{\mathbb{S}^{n-1}}(\varphi_u\,;u_0)-\mathscr{F}_{\mathbb{S}^{n-1}}(P_K \varphi_u + F(P_K\varphi_u)\,;u_0)\big)\\ 
    &\quad+\frac{1}{n-2}\mathscr{F}_{\mathbb{S}^{n-1}}(P_K \varphi_u + F(P_K\varphi_u)\,;u_0)\Big)\\
    &=\frac{1}{n - 2}\big(\mathscr{F}_{\mathbb{S}^{n-1}}(\varphi_u^{\perp}\,;u_0) + f(b)\big).
\end{align*}
We have two cases.
\begin{enumerate}[(a)]
    \item First, assume $\lvert f(b)\rvert<\frac{(n-2)C_{u_0}}{4}\Vert \varphi_{u}^{\perp} \Vert_{W^{1, 2}(\mathbb{S}^{n - 1})}^2$, for some universal constant $\nu>0$ depending only on $\tilde{u}_0$ and the dimension $n$. In this case, let $\varepsilon_f = 0$, so that $\eta \equiv 0$, and $\operatorname{IV} = 0$. Then, since
    \begin{align*}
        \mathscr{F}_{\mathbb{S}^{n-1}}(\varphi_u^{\perp}\,;u_0)&=\mathscr{F}_{\mathbb{S}^{n-1}}(\varphi_u^{\perp}\,;u_0)-\mathscr{F}_{\mathbb{S}^{n-1}}(0\,;u_0)\\
        &=\nabla\mathscr{F}_{\mathbb{S}^{n-1}}(0\,;u_0)[\varphi_u^{\perp}]+\nabla^2\mathscr{F}_{\mathbb{S}^{n-1}}(0\,;u_0)[\varphi_u^{\perp},\varphi_u^{\perp}]\\
        &=\nabla^2\mathscr{F}_{\mathbb{S}^{n-1}}(\xi\,;u_0)[\varphi_u^{\perp},\varphi_u^{\perp}],
    \end{align*}
    for some $\xi\in\mathcal{B}_{\delta_0}(0)$, by choosing $\delta_0>0$ sufficiently small we have
    \begin{align*}
        \lvert\mathscr{F}_{\mathbb{S}^{n-1}}(\varphi_u^{\perp}\,;u_0)\rvert&\le(\min_{\lambda_j > 0} \lambda_j +\max_{\lambda_j < 0} \lambda_j\Big)\|\varphi_u^{\perp}\|_{W^{1,2}(\mathbb{S}^{n-1})}^2\\
        &\le\frac{(n-2)C_{u_0}}{4}\|\varphi_u^{\perp}\|_{W^{1,2}(\mathbb{S}^{n-1})}^2.
    \end{align*}
    Then, we get
    \begin{align*}
        \lvert\mathscr{E}_{\mathbb{B}^n}(\tilde u\,;\tilde u_0)\rvert\le\frac{C_{u_0}}{2}\|\varphi_u^{\perp}\|_{W^{1,2}(\mathbb{S}^{n-1})}^2\le 1,
    \end{align*}
    where the last inequality follows by possibly choosing $\delta_0>0$ even smaller depending just on $u_0$. In particular, from 
    \eqref{equation: starting estimate}, \eqref{equation: estimate for II} and   \eqref{equation: final integral estimate for III}, we deduce 
    \begin{align*}
        \mathscr{E}_{\mathbb{B}^n}(\hat u\,;\tilde u_0)- \mathscr{E}_{\mathbb{B}^n}(\tilde u\,;\tilde u_0) &\leq - C_{u_0} \varepsilon \|\varphi_u^{\perp}\|_{W^{1,2}(\mathbb{S}^{n-1})}^2 + \hat C\varepsilon^2\|\varphi_{u}^{\perp}\|_{W^{1,2}(\mathbb{S}^{n-1})}^2\\
        & \leq -  \varepsilon (C_{u_0}- \hat C\varepsilon)\Vert \varphi_{u}^{\perp} \Vert_{W^{1, 2}(\mathbb{S}^{n - 1})}^2\\
        &\le - \varepsilon\,\frac{2(C_{u_0}- \hat C\varepsilon)}{C_{u_0}}\lvert\mathscr{E}_{\mathbb{B}^n}(\tilde u\,;\tilde u_0)\rvert,\\ 
        &\le - \varepsilon\,\lvert\mathscr{E}_{\mathbb{B}^n}(\tilde u\,;\tilde u_0)\rvert\\
        &\le - \varepsilon\,\lvert\mathscr{E}_{\mathbb{B}^n}(\tilde u\,;\tilde u_0)\rvert^{1+\gamma}
    \end{align*}
    for every $\gamma>0$, where the second to last inequality follows by choosing $\varepsilon$ appropriately so that 
    \begin{align*}
        \frac{2(C_{u_0}- \hat C\varepsilon)}{C_{u_0}}\ge 1 \quad\Leftrightarrow\quad \eps\le\frac{3C_{u_0}}{\hat C}
    \end{align*}
    and the last inequality follows because $\lvert\mathscr{E}_{\mathbb{B}^n}(\tilde u\,;\tilde u_0)\rvert\le 1$.
    \item Otherwise, we set $\varepsilon = \varepsilon_f\lvert f(b)\rvert^{1 - 2\gamma}$ for some $\varepsilon_f$ sufficiently small depending only on $n$ and $u_0$, allowing us to estimate $\operatorname{IV}$ as follows: 
    \begin{align*}
       \int_{0}^{1} \operatorname{IV} \rho^{n - 3} \; d\rho & \leq - f(b)^{1 - \gamma} \int_{0}^{1} \eta(\rho) \rho^{n - 3} \; d\rho \\
         & = - \varepsilon_f f(b)^{2 - 2 \gamma}C\int_{0}^{1} \sqrt{n}(1 - \rho)\rho^{n - 3} \; d\rho \\
         & \leq - \varepsilon_f f(b)^{2 - 2 \gamma}, 
    \end{align*}
    where the last inequality follows by choosing $C>0$ big enough depending only on $n$. Then, from this inequality, combined with \eqref{equation: starting estimate}, \eqref{equation: estimate for II} and \eqref{equation: final integral estimate for III} we infer
    \allowdisplaybreaks
    \begin{align*}
        \mathscr{E}_{\mathbb{B}^n}&(\hat u\,;\tilde u_0)-\mathscr{E}_{\mathbb{B}^n}(\tilde u\,;\tilde u_0) \\
         & \leq - C_{u_0} \varepsilon \|\varphi_u^{\perp}\|_{W^{1,2}(\mathbb{S}^{n-1})}^2 - \varepsilon_f f(b)^{2 - 2 \gamma} + \hat C\big(\varepsilon_f^2f(b)^{2-2\gamma}+\varepsilon^2\|\varphi_{u}^{\perp}\|_{W^{1,2}(\mathbb{S}^{n-1})}^2\big) \\
         & \leq - (C_{u_0} - \hat C\varepsilon)\eps\|\varphi_u^{\perp}\|_{W^{1,2}(\mathbb{S}^{n-1})}^2 - (\varepsilon_f - \hat C \varepsilon_{f}^{2}) f(b)^{2 - 2 \gamma}\\
         &\le -\bigg(\frac{4(C_{u_0}-\hat C\eps)}{(n-2)C_{u_0}}-1+\hat C\eps_f\bigg)\eps_f f(b)^{2-2\gamma}\\
         &\le -\bigg(\frac{2(C_{u_0}-\hat C\eps)}{C_{u_0}}-\frac{n-2}{2}+\frac{n-2}{2}\hat C\eps_f\bigg)\eps_f\lvert\mathscr{E}_{\mathbb{B}^n}(\tilde u\,;\tilde u_0)\rvert^{2-2\gamma}\\
         &\le -\bigg(\frac{2(C_{u_0}-\hat C\eps)}{C_{u_0}}-\frac{n-2}{2}+\frac{n-2}{2}\hat C\eps_f\bigg)\eps_f\lvert\mathscr{E}_{\mathbb{B}^n}(\tilde u\,;\tilde u_0)\rvert^{1+\tilde\gamma},
    \end{align*}
    where in the last inequality we have defined $\tilde\gamma:=1-2\gamma\in[0,1)$. Again, by choosing $\eps_f$ (and hence $\eps$) small enough so that 
    \begin{align*}
        \frac{2(C_{u_0}-\hat C\eps)}{C_{u_0}}-\frac{n-2}{2}+\frac{n-2}{2}\hat C\eps_f\ge 1,
    \end{align*}
    we get that
    \begin{align*}
        \mathscr{E}_{\mathbb{B}^n}&(\hat u\,;\tilde u_0)-\mathscr{E}_{\mathbb{B}^n}(\tilde u\,;\tilde u_0)\le -\eps_f\lvert\mathscr{E}_{\mathbb{B}^n}(\tilde u\,;\tilde u_0)\rvert^{1+\tilde\gamma}
    \end{align*}
    and the statement follows. 
\end{enumerate}

\subsection{The integrable case} \label{subsec: integrable case}
We now specialise the proof of Theorem \ref{Theorem: (log)-epiperimetric inequality for harmonic maps} to the case of an integrable kernel (of the second variation). We start by recalling from \cite{AdamsSimon} this notion. We will say that $K:=\ker\nabla^2\mathscr{F}_{\mathbb{S}^{n-1}}(0\,;u_0)$ is integrable if for every $v \in K$, there exists a family $\{u_s\}_{s \in (0, 1)} \subset C^{\infty}(u_0^*TN)$ with $u_s \rightarrow 0$ in $C^{\infty}(u_0^*TN)$, such that $ \nabla\mathscr{F}_{M}(u_s\,;u_0) = 0$ for every $s \in (0, 1)$, and $\lim_{s \rightarrow 0} u_s/s = v$ in the $L^2(M)$ sense. In this setting, analyticity of $f$ defined above implies the following lemma, whose proof can be found in \cite[Lemma 1]{AdamsSimon}, or \cite[Lemma 2.3]{EngelsteinSpolaorVelichkov}. 
\begin{lemma}
    The integrability condition hold for $\ker\nabla^2\mathscr{F}_{\mathbb{S}^{n-1}}(0\,;u_0)$ if and only if $f \equiv f(0) = 0$ in a neighborhood of $0$. 
\end{lemma}
Consequently, it is immediate from this lemma that in the proof of the log-epiperimetric inequality we can take $\gamma = 0$, thus obtaining an epiperimetric inequality. 
\section{Proof of the uniqueness of tangent maps with isolated singularities (Theorem \ref{Theorem: uniqueness of tangent maps with isolated singularities})} \label{sec: proof of uniqueness}
\allowdisplaybreaks
As by the statement of Theorem \ref{Theorem: uniqueness of tangent maps with isolated singularities}, consider an energy minimizing harmonic map $u\in W^{1,2}(\Omega,N)$ and let $y\in\operatorname{Sing}(u)$. Let $r>0$ be such that $B_r(y)\subset\Omega$. For every $\rho\in(0,r/2)$, define $u_{y,\rho}\in W^{1,2}(B_2(0),N)$ as
\begin{align*}
    u_{y,\rho}(x):=u(\rho x+y) \qquad\mbox{ for } \mathcal{L}^n\mbox{-.a.e. } x\in B_2(0).
\end{align*}
First, we show the following lemma that will be used to analyze the rescalings $u_{y,\rho}$ at comparable scales.
\begin{lemma}\label{Lemma: blow-ups at comparable scales}
    For every $\varepsilon_0>0$ there exists $\delta_0>0$ such that for every $r\in(0,\delta_0)$ and every $\rho\in[r/2,r]$ we have 
    \begin{align*}
        \int_{B_{\frac{3}{2}}(0)\smallsetminus B_{\frac{3}{4}}(0)}\lvert u_{y,\rho}-u_{y,r}\rvert^2\, d\mathcal{L}^{n}<\varepsilon_0.
    \end{align*}
    \begin{proof}
        We argue by contradiction and we claim that there exist $\varepsilon_0>0$ and sequences $\{r_n\}_{n\in\N}$ and $\{\rho_n\}_{n\in\N}$ with $r_n\to 0$ as $n\to+\infty$ and $\rho_n\in[r_n/2,r_n]$ such that
        \begin{align}\label{Equation: contradiction}
            \int_{B_{\frac{3}{2}}(0)\smallsetminus B_{\frac{3}{4}}(0)}\lvert u_{y,\rho_n}-u_{y,r_n}\rvert^2\, d\mathcal{L}^{n}\ge\varepsilon_0.
        \end{align}
        Since 
        \begin{align*}
            1\le\frac{r_n}{\rho_n}\le 2
        \end{align*}
        there exists a subsequence (not relabeled) such that
        \begin{align*}
            0<\ell:=\lim_{n\to+\infty}\frac{r_n}{\rho_n}<+\infty.
        \end{align*}
        By the standard monotonicity formula for energy minimizing harmonic maps we know that, up to a subsequence (not relabelled) $u_{y,r_n}\rightharpoonup\varphi$ weakly in $W^{1,2}$ and strongly in $L^2$ for some $0$-homogeneous tangent map $\varphi\in W^{1,2}(B_2(0),N)$. Let $\Psi_{\ell}:\R^n\to\R^n$ be given by
        \begin{align*}
            \Psi_{\ell}(x):=\ell x \qquad\forall\,x\in\R^n
        \end{align*}
        and notice that
        \begin{align*}
            \lim_{n\to+\infty}\int_{B_{\frac{3}{2}}(0)\smallsetminus B_{\frac{3}{4}}(0)}\lvert u_{y,\rho_n}-\varphi\rvert^2\, d\mathcal{L}^n&=\lim_{n\to+\infty}\int_{B_{\frac{3}{2}}(0)\smallsetminus B_{\frac{3}{4}}(0)}\lvert(\Psi_{\ell}^{-1})^*(u_{y,\ell\rho_n}-\Psi_{\ell}^{*}\varphi)\rvert^2\, d\mathcal{L}^n\\
            &\le C\lim_{n\to+\infty}\int_{B_{\frac{3}{2}}(0)\smallsetminus B_{\frac{3}{4}}(0)}\lvert u_{y,r_n}-\varphi\rvert^2\, d\mathcal{L}^n=0.
        \end{align*}
        By then, by triangle inequality we get 
        \begin{align*}
            &\lim_{n\to+\infty}\int_{B_{\frac{3}{2}}(0)\smallsetminus B_{\frac{3}{4}}(0)}\lvert u_{y,\rho_n}-u_{u,r_n}\rvert^2\, d\mathcal{L}^{n}\\
            &\le 2\lim_{n\to+\infty}\bigg(\int_{B_{\frac{3}{2}}(0)\smallsetminus B_{\frac{3}{4}}(0)}\lvert u_{y,\rho_n}-\varphi\rvert^2\, d\mathcal{L}^{n}+\int_{B_{\frac{3}{2}}(0)\smallsetminus B_{\frac{3}{4}}(0)}\lvert u_{y,r_n}-\varphi\rvert^2\, d\mathcal{L}^{n}\bigg)=0
        \end{align*}
        which contradicts \eqref{Equation: contradiction}. The statement follows. 
    \end{proof}
\end{lemma}
Let $\varphi\in W^{1,2}(B_2(0),N)$ be any tangent map for $u$ at $y$ such that $\varphi\in C^{\infty}(B_2(0)\smallsetminus\{0\},N)$. Let $\{\rho_i\}_{i\in\N}$ be such that $u_{y,\rho_i}\rightharpoonup\varphi$ weakly in $W^{1,2}$. Fix any $\eta>0$. By Rellich-Kondrachov theorem, we get that $u_{y,\rho_i}\to\varphi$ strongly in $L^2$. Thus, there exists $i_1 \in\N$ such that for every $i\ge i_1$ we have
\begin{align*}
    \int_{B_{\frac{3}{2}}(0)\smallsetminus B_{\frac{3}{4}}(0)}\lvert u_{y,\rho_i}-\varphi\rvert^2\, d\mathcal{L}^{n}<\eta.
\end{align*}
Let $\delta_0>0$ be the constant given by Lemma \ref{Lemma: blow-ups at comparable scales} with $\varepsilon_0=\eta$ and let $i_2\in\N$ be such that $\rho_i<\delta_0$ for every $i\ge i_2$. Let $\tilde i:=\max\{i_1,i_2\}$ and let $\tilde\rho:=\rho_{\tilde i}$. Fix any $n\in\N$. Then, by Lemma \ref{Lemma: blow-ups at comparable scales}, for every $\rho\in[\tilde\rho/2^{n+1},\tilde\rho/2^n]$ we have 
\begin{align*}
    \int_{B_{\frac{3}{2}}(0)\smallsetminus B_{\frac{3}{4}}(0)}\lvert u_{y,\rho}-\varphi\rvert^2\, d\mathcal{L}^n&\le 2\Bigg(\int_{B_{\frac{3}{2}}(0)\smallsetminus B_{\frac{3}{4}}(0)}\lvert u_{y,\rho}-u_{y,\frac{\tilde\rho}{2^n}}\rvert^2\, d\mathcal{L}^n\\
    &\quad+\int_{B_{\frac{3}{2}}(0)\smallsetminus B_{\frac{3}{4}}(0)}\lvert u_{y,\frac{\tilde\rho}{2^n}}-\varphi\rvert^2\, d\mathcal{L}^n\Bigg)\\
    &\le 4\eta.
\end{align*}
Notice that $u_{y,\rho}-\varphi$ is an energy minimizing harmonic map and let $\eps_0>0$ be the constant given by the $\eps$-regularity theorem for energy minimizing harmonic maps, cf. \cite{schoen-uhlenbeck}. By choosing $\eta>0$ small enough depending just on $\eps_0$ and on $\varphi$ we get that $u_{y,\rho}\in C^{\infty}(B_{\frac{5}{4}}(0)\smallsetminus B_{\frac{7}{8}}(0),N)$ and 
\begin{align*}
    \|u_{y,\rho}-\varphi\|_{C^{2,\alpha}(\mathbb{S}^{n-1})}\le \|u_{y,\rho}-\varphi\|_{C^{3}\big(B_{\frac{5}{4}}(0)\smallsetminus B_{\frac{7}{8}}(0)\big)}<C\eta.
\end{align*}
Let $\varepsilon,\delta>0$ and $\gamma\in[0,1)$ be the constants given by Theorem \ref{Theorem: (log)-epiperimetric inequality for harmonic maps} for $u_0=\varphi$. By again reducing the size of $\eta>0$ we can make sure that $C\eta<\delta$, so that 
\begin{align*}
    \|u_{y,\rho}-\varphi\|_{C^{2,\alpha}(\mathbb{S}^{n-1})}\le C\eta<\delta.
\end{align*}
Thus, by Theorem \ref{Theorem: (log)-epiperimetric inequality for harmonic maps}, there exists $\hat u_{\rho}\in W^{1,2}(\mathbb{B}^n,N)$ such that $\hat u_{\rho}|_{\mathbb{S}^{n-1}}=u_{y,\rho}|_{\mathbb{S}^{n-1}}$ and 
\begin{align*}
    \mathscr{E}_{\mathbb{B}^n}(\hat u_{\rho}\,;\varphi)\le\big(1-\varepsilon\lvert\mathscr{E}_{\mathbb{B}^n}(\tilde u_{\rho}\,;\varphi)\rvert^{\gamma}\big)\mathscr{E}_{\mathbb{B}^n}(\tilde u_{\rho}\,;\varphi),
\end{align*}
where $\tilde u_{\rho}\in W^{1,2}(\mathbb{B}^n,N)$ is the $0$-homogeneous extension of $u_{y,\rho}$ inside $\mathbb{B}^n$. Notice that, since $u_{y,\rho}$ is energy minimizing, we have 
\begin{align}\label{Equation: estimate epiperimetric}
\begin{split}
    \Theta(\rho,y;u)&-\Theta(y;u)=\mathscr{D}_{\mathbb{B}^n}(u_{y,\rho})-\Theta(y;u)\\
    &=\mathscr{D}_{\mathbb{B}^n}(u_{y,\rho})-\mathscr{D}_{\mathbb{B}^n}(\varphi)\\
    &\le\mathscr{D}_{\mathbb{B}^n}(\hat u_{\rho})-\mathscr{D}_{\mathbb{B}^n}(\varphi)\\
    &=\mathscr{E}_{\mathbb{B}^n}(\hat u_{\rho}\,;\varphi) \le \big(1-\varepsilon\lvert\mathscr{E}_{\mathbb{B}^n}(\tilde u_{\rho}\,;\varphi)\rvert^{\gamma}\big)\mathscr{E}_{\mathbb{B}^n}(\tilde u_{\rho}\,;\varphi) \qquad\forall\,\rho\in(\tilde\rho/2^{n+1},\tilde\rho/2^n)
\end{split}
\end{align}
Let
\begin{align*}
    f(\rho):=\rho^{n-2}\Theta(\rho,y;u)-\Theta(y;u)\rho^{n-2}=\int_{B_{\rho}(y)}\lvert du\rvert^2\, d\mathcal{L}^n-\Theta(y;u)\rho^{n-2} \qquad\forall\,\rho\in[0,1).
\end{align*}
Notice that by the monotonicity formula for energy minimizing harmonic maps we have that $[0,1)\ni\rho\mapsto f(\rho)$ is an non-decreasing function of $\rho$. Hence, $f$ is differentiable $\mathcal{L}^1$-a.e. and its distributional derivative is a measure whose absolutely continous part (with respect to $\mathcal{L}^1$) coincides $\mathcal{L}^1$-a.e. with the classical differential and whose singular part is non negative. Thus, we have
\begin{align*}
    f'(\rho)&\ge\int_{\partial B_{\rho}(y)}\lvert du\rvert^2\, d\mathscr{H}^{n-1}-(n-2)\Theta(y;u)\rho^{n-3}\\
    &=\rho^{n-1}\int_{\mathbb{S}^{n-1}}\lvert du(\rho x+y)\rvert^2\, d\mathscr{H}^{n-1}(x)-(n-2)\Theta(y;u)\rho^{n-3}\\
    &=\rho^{n-3}\Bigg(\int_{\mathbb{S}^{n-1}}\lvert \rho du(\rho x+y)\rvert^2\, d\mathscr{H}^{n-1}(x)-(n-2)\Theta(y;u)\Bigg)\\
    &=\rho^{n-3}\Bigg(\int_{\mathbb{S}^{n-1}}\lvert du_{y,\rho}\rvert^2\, d\mathscr{H}^{n-1}(x)-(n-2)\Theta(y;u)\Bigg)\\
    &=\rho^{n-3}(n-2)\big(\mathscr{D}_{\mathbb{B}^n}(\tilde u_{\rho})-\mathscr{D}_{\mathbb{B}^n}(\varphi)\big)\\
    &=\rho^{n-3}(n-2)\mathscr{E}_{\mathbb{B}^n}(\tilde u_{\rho}\,;\varphi) \qquad\mbox{ for } \mathcal{L}^1\mbox{-a.e. } \rho\in(0,1),
\end{align*}
which can be rewritten as
\begin{align}\label{Equation: estimate epiperimetric 2}
    \rho^{n-2}\mathscr{E}_{\mathbb{B}^n}(\tilde u_{\rho}\,;\varphi)\le\frac{\rho}{n-2}f'(\rho) \qquad\mbox{ for } \mathcal{L}^1\mbox{-a.e. } \rho\in(0,1).
\end{align}
By plugging \eqref{Equation: estimate epiperimetric 2} in \eqref{Equation: estimate epiperimetric} we get
\begin{align}\label{Equation: epiperimetric 3}
\begin{split}
    f(\rho)&=\rho^{n-2}\Theta(\rho,y;u)-\Theta(y;u)\rho^{n-2}\\
    &\le\big(1-\varepsilon\lvert\mathscr{E}_{\mathbb{B}^n}(\tilde u_{\rho}\,;\varphi)\rvert^{\gamma}\big)\rho^{n-2}\mathscr{E}_{\mathbb{B}^n}(\tilde u_{\rho}\,;\varphi)\\
    &\le\big(1-\varepsilon\lvert\mathscr{E}_{\mathbb{B}^n}(\tilde u_{\rho}\,;\varphi)\rvert^{\gamma}\big)\frac{\rho}{n-2}f'(\rho), \qquad\mbox{ for } \mathcal{L}^1\mbox{-a.e. } \rho\in(\tilde\rho/2^{n+1},\tilde\rho/2^n).
\end{split}
\end{align}
Moreover, since $u_{y,\rho}$ is energy minimizing, we have 
\begin{align}\label{Equation: epiperimetric 4}
\begin{split}
    e(\rho)&:=\Theta(\rho,y;u)-\Theta(y;u)=\frac{1}{\rho^{n-2}}\int_{B_{\rho}(y)}\lvert du\rvert^2\, d\mathcal{L}^n-\Theta(y;u)=\mathscr{D}_{\mathbb{B}^n}(u_{y,\rho})-\Theta(y;u)\\
    &\le\mathscr{D}_{\mathbb{B}^n}(\tilde u_{\rho})-\Theta(y;u)=\mathscr{D}_{\mathbb{B}^n}(\tilde u_{\rho})-\mathscr{D}_{\mathbb{B}^n}(\varphi)=\mathscr{E}_{\mathbb{B}^{n}}(\tilde u_{\rho}\,;\varphi).
\end{split}
\end{align}
Hence, by combining \eqref{Equation: epiperimetric 3} and \eqref{Equation: epiperimetric 4} we get
\begin{align*}
    f(\rho)\le\big(1-\varepsilon\lvert e(\rho)\rvert^{\gamma}\big)\frac{\rho}{n-2}f'(\rho), \qquad\mbox{ for } \mathcal{L}^1\mbox{-a.e. } \rho\in(\tilde\rho/2^{n+1},\tilde\rho/2^n).
\end{align*}
Arguing as in \cite[Section 3.2, Step 1]{EngelsteinSpolaorVelichkov} we get
\begin{align*}
    e(\rho)\le 2\bigg(-(n-2)\varepsilon\gamma\log\bigg(\frac{\rho}{\tilde\rho}\bigg)\bigg)^{-\frac{1}{\gamma}} \qquad\forall\,\rho\in[\tilde\rho/2^{n+1},\tilde\rho/2^n].
\end{align*}
Since we have chosen $n\in\N$ arbitrarily and for every $\rho \in (0,\tilde\rho)$ there exists $n\in\N$ such that $\rho\in[\tilde\rho/2^{n+1},\tilde\rho/2^n]$, we have established that
\begin{align}\label{Equation: logarithmic decay of the energy density}
    \Theta(\rho,y;u)-\Theta(y;u)=e(\rho)\le 2\bigg(\hspace{-1.5mm}-(n-2)\varepsilon\gamma\log\bigg(\frac{\rho}{\tilde\rho}\bigg)\bigg)^{-\frac{1}{\gamma}} \qquad\forall\,\rho\in(0,\tilde\rho).
\end{align}
The uniqueness of tangent map to $u$ at $y$ then follows directly by Proposition \ref{Proposition: Dini decrease implies uniqueness of tangent map} with $\rho_0:=\tilde\rho/2$ and
\begin{align*}
    \phi(\rho):=2\bigg(\hspace{-1.5mm}-(n-2)\varepsilon\gamma\log\bigg(\frac{\rho}{\tilde\rho}\bigg)\bigg)^{-\frac{1}{\gamma}} \qquad\forall\,\rho\in(0,\tilde\rho/2).
\end{align*}

\section{Proof of the uniqueness of tangent maps at infinity (Theorem \ref{thm: uniqueness of blow-downs})}\label{sec: uniqueness of tangents at infinity} 

In this section, we aim to prove the uniqueness of the blow-down maps, i.e. uniqueness of the tangent map at infinity. We start by recalling what these objects are. 
\begin{definition}[Tangent maps at infinity]\label{Definition: tangent maps at infinity}
    Let $N\subset\R^k$ be a closed smooth submanifold in $\R^k$ and let $n\ge 3$. Let $u\in W_{loc}^{1,2}(\R^n,N)$. For every $\rho\in[1,+\infty)$, define $u_{\rho}\in W^{1,2}(B_2(0),N)$ as
    \begin{align*}
        u_{\rho}(x):=u_{0,\rho}(x)=u(\rho x) \qquad\mbox{ for } \mathcal{L}^n\mbox{-.a.e. } x\in B_2(0).
    \end{align*}
    If there exist a sequence $\{\rho_i\}_{i\in\N}\subset[1,+\infty)$ and a map $\varphi\in W^{1,2}(B_2(0),N)$ such that $\rho_i\to+\infty$ and $u_{\rho_i}\rightharpoonup\varphi$ weakly in $W^{1,2}(B_2(0),N)$, then we say that $\varphi$ is a \textit{blow-down} of $u$, or a \textit{tangent map for $u$ at infinity}.
\end{definition}
\begin{remark}[Existence of blow downs]\label{Remark: blow-downs of u}
    Let $u\in W_{loc}^{1,2}(\R^n,N)$ and assume the following growth condition on the Dirichlet energy of $u$ at infinity: there exists $\Lambda>0$ such that
    \begin{align}\label{Equation: growth at infinity}
        \int_{B_{\rho}(0)}\lvert du\rvert^2\, d\mathcal{L}^n\le\Lambda\rho^{n-2} \qquad\forall\,\rho\in (0,+\infty). 
    \end{align}
    This implies that the Dirichlet energy of the blow-downs $\{u_{\rho}\}_{\rho\in[1,+\infty)}$ of $u$ is uniformly bounded. Indeed, we have
    \begin{align*}
        \mathscr{D}_{\mathbb{B}^n}(u_{\rho})=\frac{1}{\rho^{n-2}}\int_{B_{\rho}(0)}\lvert du\rvert^2\, d\mathcal{L}^n\le\Lambda<+\infty.
    \end{align*}
    Hence, by standard Sobolev weak compactness, there exist a sequence $\{\rho_i\}_{i\in\N}\subset[1,+\infty)$ and a map $\varphi\in W^{1,2}(B_2(0),N)$ such that $\rho_i\to+\infty$ and $u_{\rho_i}\rightharpoonup\varphi$ weakly in $W^{1,2}(B_2(0),N)$. Hence, $u$ admits at least a tangent map at infinity.
\end{remark}
Now, we aim to provide a proof of Theorem \ref{thm: uniqueness of blow-downs}, i.e. to show that if $u$ is an energy minimizing harmonic map satisfying \eqref{Equation: growth at infinity}, then $u$ has a unique tangent map at infinity. 

Note that, because of the monotonicity formula for energy minimizing harmonic maps, the function 
    \begin{align*}
        [1,+\infty)\ni\rho\mapsto\Theta(\rho\,;u):=\mathscr{D}_{\mathbb{B}^n}(u_{\rho})=\frac{1}{\rho^{n-2}}\int_{B_{\rho}(0)}\lvert du\rvert^2\, d\mathcal{L}^n
    \end{align*}
    is a non-decreasing function on $[1,+\infty)$. By Remark \ref{Remark: blow-downs of u}, the function $\Theta(\,\cdot\,;u)$ is also uniformly bounded on $[1,+\infty)$ by the constant $\Lambda$. Hence,
    \begin{align*}
        \lim_{\rho\to+\infty}\Theta(\rho\,;u)        
    \end{align*}
    exists and is a finite non-negative real number, which we denote by $\Theta(u)$.\\
    Let $\varphi\in W^{1,2}(B_{2}(0)\,;N)$ be any tangent map to $u$ at infinity and let $\{\rho_i\}_{i\in\N}\subset[1,+\infty)$ be such that $\rho_i\to+\infty$ and $u_{\rho_i}\rightharpoonup\varphi$ weakly in $W^{1,2}(B_2(0),N)$. Note that, since $u$ is energy minimizing, subsequentially we have $u_{\rho}\to\varphi$ strongly in $W^{1,2}(\mathbb{B}^n,N)$ (see e.g. \cite[Section 3.9, Lemma 1]{simon-book}) and thus the we have
    \begin{align*}
        \mathscr{D}_{\mathbb{B}^n}(\varphi)=\Theta(u).
    \end{align*}
    Hence, we have
    \begin{align*}
        \mathscr{E}_{\mathbb{B}^n}(u_{\rho}\,;\varphi)=\mathscr{D}_{\mathbb{B}^n}(u_{\rho})-\mathscr{D}_{\mathbb{B}^n}(\varphi)=\Theta(\rho\,;u)-\Theta(u)
    \end{align*}
    and by the monotonicity formula we get $\mathscr{E}_{\mathbb{B}^n}(u_{\rho}\,;\varphi)\le 0$ for every $\rho\in[1,+\infty)$. Thus, we got that
    \begin{align*}
        [1,+\infty)\ni\rho\mapsto\mathscr{E}_{\mathbb{B}^n}(u_{\rho}\,;u_{\infty})=\Theta(\rho\,;u)-\Theta(u)\in[\Theta(1\,;u)-\Theta(u),0]\subset(-\infty,0]
    \end{align*}
    is a non-decreasing and non-positive function on $[1,+\infty)$. As such, it is differentiable $\mathcal{L}^1$-a.e. with respect to $\rho$ and we have 
    \begin{align*}
        \frac{d}{d\rho}\mathscr{E}_{\mathbb{B}^n}(u_{\rho}\,;\varphi)&\ge-\frac{n-2}{\rho^{n-1}}\int_{B_{\rho}(0)}\lvert du\rvert^2\, d\mathcal{L}^n+\frac{1}{\rho^{n-2}}\int_{\partial B_{\rho}(0)}\lvert du\rvert^2\, d\mathscr{H}^{n-1}\\
        &=\frac{n-2}{\rho}\big(\mathscr{E}_{\mathbb{B}^n}(\tilde u_{\rho}\,;\varphi)-\mathscr{E}_{\mathbb{B}^n}(u_{\rho}\,;\varphi)\big)\\
        &\quad-\frac{n-2}{\rho}\mathscr{D}_{\mathbb{B}^n}(\tilde u_{\rho})+\frac{1}{\rho^{n-2}}\int_{\partial B_{\rho}(0)}\lvert du\rvert^2\, d\mathscr{H}^{n-1}\\
        &=\frac{n-2}{\rho}\big(\mathscr{E}_{\mathbb{B}^n}(\tilde u_{\rho}\,;\varphi)-\mathscr{E}_{\mathbb{B}^n}(u_{\rho}\,;\varphi)\big)\\
        &\quad-\frac{1}{\rho}\int_{\mathbb{S}^{n-1}}\lvert du_{\rho}\rvert^2\,d\mathscr{H}^{n-1}+\frac{1}{\rho^{n-2}}\int_{\partial B_{\rho}(0)}\lvert du\rvert^2\, d\mathscr{H}^{n-1}\\
        &=\frac{n-2}{\rho}\big(\mathscr{E}_{\mathbb{B}^n}(\tilde u_{\rho}\,;\varphi)-\mathscr{E}_{\mathbb{B}^n}(u_{\rho}\,;\varphi)\big)\\
        &\quad-\frac{1}{\rho^{n-2}}\int_{\partial B_{\rho}(0)}\lvert du\rvert^2\,d\mathscr{H}^{n-1}+\frac{1}{\rho^{n-2}}\int_{\partial B_{\rho}(0)}\lvert du\rvert^2\, d\mathscr{H}^{n-1}\\
        &=\frac{n-2}{\rho}\big(\mathscr{E}_{\mathbb{B}^n}(\tilde u_{\rho}\,;\varphi)-\mathscr{E}_{\mathbb{B}^n}(u_{\rho}\,;\varphi)\big)\ge 0,
    \end{align*} 
    where the last inequality follows from the fact that $u$ is energy minimizing. We have then shown property (1) in \cite[Assumptions 2.1]{SymmetricEdelenSpolaorVelichkov}. Note that property (2) in \cite[Assumptions 2.1]{SymmetricEdelenSpolaorVelichkov} is exactly the symmetric log-epiperimetric inequality 
    \begin{align*}
		\mathscr{E}_{\mathbb{B}^n}(\hat u\,;\tilde u_0)\le\mathscr{E}_{\mathbb{B}^n}(\tilde u\,;\tilde u_0)-\varepsilon\lvert\mathscr{E}_{\mathbb{B}^n}(\tilde u\,;\tilde u_0)\rvert^{1+\gamma},
	\end{align*}
    that we have shown in Theorem \ref{Theorem: (log)-epiperimetric inequality for harmonic maps}. Then we can apply \cite[Corollary 2.6]{SymmetricEdelenSpolaorVelichkov} or argue directly exactly as in Section 5 and Appendix A to obtain that there exists $\alpha>0$ such that
    \begin{align*}
        \|u_{\rho}-\varphi\|_{L^2(\mathbb{S}^{n-1})}\le\begin{cases}
            C\log(\rho)^{\frac{\gamma-1}{2\gamma}} & \mbox{ if } \gamma>0\\
            C\rho^{-\alpha} & \mbox{ if } \gamma=0
        \end{cases} \qquad\forall\rho\in (1,+\infty).
    \end{align*}
    This immediately implies that $\varphi$ is the unique limit of the blow-down family $\{u_{\rho}\}_{\rho\in[1,+\infty)}$, i.e. the unique tangent map of $u$ at infinity.
\appendix
\section{A criterion for the uniqueness of tangent maps} \label{sec: appendix criterion}
\allowdisplaybreaks
The aim of this last section is to prove a standard argument in the literature allowing us to infer uniqueness of tangent maps to a stationary harmonic map $u$ at some point $y$ from a sufficiently fast decay of the energy density $\Theta(\rho,y;u)$ to its limit, usually referred to as \textit{Dini continuity}. We reproduce the argument here for the sake of completeness. 
\begin{proposition}\label{Proposition: Dini decrease implies uniqueness of tangent map}
    Let $N\subset\R^k$ be a closed smooth submanifold in $\R^k$ and let $\Omega\subset\mathbb{R}^n$ be any open set. Let $u\in W^{1,2}(\Omega,N)$ be a stationary harmonic map on $\Omega$ and let $y\in\Omega$. Assume that there exist $\rho_0\in(0,\dist(y,\partial\Omega))$ and an increasing function $\phi:(0,\rho_0)\to(0,+\infty)$ such that 
    \begin{align}\label{Equation: assumptions bound}
        \Theta(\rho,y;u)-\Theta(y;u)\le\phi(\rho) \qquad\forall\,\rho\in(0,\rho_0)
    \end{align}
    and 
    \begin{align}\label{Equation: integrability condition around the origin}
        \int_0^{\rho_0}\frac{\sqrt{\phi(\rho)}}{\rho}d\mathcal{L}^1(\rho)<+\infty.
    \end{align}
    Then, the tangent map to $u$ at $y$ is unique.
    \end{proposition}
    \begin{proof}
        First, recall the monotonicity formula for stationary harmonic maps (whose proof can be found for instance in \cite[Section 1, Equation (1.7)]{lin-gradient-estimates}):
        \begin{align}\label{Equation: monotonicity for stationary harmonic maps}
            \frac{1}{\rho^{n-2}}\int_{B_{\rho}(y)}\lvert du\rvert^2\,d\mathcal{L}^{n}-\frac{1}{\sigma^{n-2}}\int_{B_{\sigma}(y)}\lvert du\rvert^2\,d\mathcal{L}^{n}=\int_{B_{\rho}(y)\smallsetminus B_{\sigma}(y)}\frac{1}{\lvert\,\cdot\,-y\rvert^{n-2}}\bigg\lvert\frac{\partial u}{\partial\nu_y}\bigg\rvert^2\, d\mathcal{L}^n
        \end{align}
        for every $0<\sigma<\rho<\dist(y,\partial\Omega)$, where we have defined
        \begin{align*}
            \nu_{y}:=\frac{\,\cdot\,-y}{\lvert\,\cdot\,-y\rvert} \qquad\mbox{ on } \Omega\smallsetminus\{y\}.
        \end{align*}
        By H\"older inequality, \eqref{Equation: monotonicity for stationary harmonic maps} and the bound \eqref{Equation: assumptions bound} in the assumptions, for every $0<\sigma<\rho<\rho_0$ we have
        \begin{align*}
            \int_{B_{\rho}(y)\smallsetminus B_{\sigma}(y)}&\frac{1}{\lvert\,\cdot\,-y\rvert^{n-1}}\bigg\lvert\frac{\partial u}{\partial\nu_y}\bigg\rvert\, d\mathcal{L}^n\\
            &\le\bigg(\int_{B_{\rho}(y)\smallsetminus B_{\sigma}(y)}\frac{1}{\lvert\,\cdot\,-y\rvert^{n-2}}\bigg\lvert\frac{\partial u}{\partial\nu_y}\bigg\rvert^2\, d\mathcal{L}^n\bigg)^{\frac{1}{2}}\bigg(\int_{B_{\rho}(y)\smallsetminus B_{\sigma}(y)}\frac{1}{\lvert\,\cdot\,-y\rvert^{n}}\, d\mathcal{L}^n\bigg)^{\frac{1}{2}}\\
            &\le C\big((\log\rho-\log\sigma)(\Theta(\rho,y;u)-\Theta(y;u))\big)^{\frac{1}{2}}\\
            &\le C\big((\log\rho-\log\sigma)\phi(\rho)\big)^{\frac{1}{2}}.
        \end{align*}
        Fix any $0<\sigma<\rho<\rho_0$ and let $k\in\N$ be such that  $\rho/2^k\le\sigma$. From the previous estimate, for every $i\in\N$ we get 
        \begin{align*}
             \int_{B_{\rho/2^{i}}(y)\smallsetminus B_{\rho/2^{i+1}}(y)}\frac{1}{\lvert\,\cdot\,-y\rvert^{n-1}}\bigg\lvert\frac{\partial u}{\partial\nu_y}\bigg\rvert\, d\mathcal{L}^n\le \tilde C\sqrt{\phi\bigg(\frac{\rho}{2^i}\bigg)}
        \end{align*}
        Then we have
        \begin{align}\label{Equation: estimate uniqueness of blow ups}
        \begin{split}
             \int_{B_{\rho}(y)\smallsetminus B_{\sigma}(y)}\frac{1}{\lvert\,\cdot\,-y\rvert^{n-1}}\bigg\lvert\frac{\partial u}{\partial\nu_y}\bigg\rvert\, d\mathcal{L}^n&\le\int_{B_{\rho}(y)\smallsetminus B_{\rho/2^{k}}(y)}\frac{1}{\lvert\,\cdot\,-y\rvert^{n-1}}\bigg\lvert\frac{\partial u}{\partial\nu_y}\bigg\rvert\, d\mathcal{L}^n\\
             &=\sum_{i=0}^{k-1}\int_{B_{\rho/2^{i}}(y)\smallsetminus B_{\rho/2^{i+1}}(y)}\frac{1}{\lvert\,\cdot\,-y\rvert^{n-1}}\bigg\lvert\frac{\partial u}{\partial\nu_y}\bigg\rvert\, d\mathcal{L}^n\\
             &\le\sum_{i=0}^{n-1}\sqrt{\phi\bigg(\frac{\rho}{2^i}\bigg)}=\sum_{i=0}^{n-1}\sqrt{\phi\bigg(\frac{\rho}{2^i}\bigg)}\frac{2^i}{\rho}\frac{\rho}{2^i}\\
             &\le\int_0^{\rho}\frac{\sqrt{\phi(t)}}{t}\,d\mathcal{L}^1(t).
        \end{split}
        \end{align}
    	Now, let $\varphi_1,\varphi_2$ be any two tangent maps to $u$ at the point $y\in\Omega$. By definition of tangent map, there exist sequences $\{\rho_i\}_{i\in\N}\subset(0,\rho_0)$ and $\{\sigma_i\}_{i\in\N}\subset(0,\rho_0)$ such that $\rho_i,\sigma_i\to 0^+$ and
    	\begin{align*}
    		&u_{y,\rho_i}\rightharpoonup\varphi_1 \qquad \text{and} \qquad u_{y,\sigma_i}\rightharpoonup\varphi_2
    	\end{align*}
    	weakly in $W^{1,2}(\mathbb{B}^n)$ as $i\to+\infty$. By the weak continuity of the trace operator, we have 
    	\begin{align*}
    		u_{y,\rho_i}|_{\mathbb{S}^{n-1}}\rightharpoonup\varphi_1|_{\mathbb{S}^{n-1}} \qquad \text{and} \qquad u_{y,\sigma_i}|_{\mathbb{S}^{n-1}}\rightharpoonup\varphi_2|_{\mathbb{S}^{n-1}}
    	\end{align*}
    	weakly in $L^2(\mathbb{S}^{n-1})$ as $i\to+\infty$.  Thus, by \eqref{Equation: estimate uniqueness of blow ups} and the coarea formula, we have
        \begin{align*}
            \int_{\mathbb{S}^{n-1}}\lvert\varphi_1-\varphi_2\rvert\, d\mathscr{H}^{n-1}&\le\liminf_{i\to+\infty}\int_{\mathbb{S}^{n-1}}\lvert u_{y,\rho_i}-u_{y,\sigma_i}\rvert\, d\mathscr{H}^{n-1}\\
            &=\liminf_{i\to+\infty}\int_{\mathbb{S}^{n-1}}\lvert u(\rho_ix+y)-u(\sigma_ix+y)\rvert\, d\mathscr{H}^{n-1}(x)\\
            &=\liminf_{i\to+\infty}\int_{\mathbb{S}^{n-1}}\bigg\lvert\int_{\sigma_i}^{\rho_i}\nabla u(\rho x+y)\cdot x\, d\mathcal{L}^1(\rho)\bigg\rvert\, d\mathscr{H}^{n-1}(x)\\
            &=\liminf_{i\to+\infty}\int_{\mathbb{S}^{n-1}}\int_{\sigma_i}^{\rho_i}\bigg\lvert\nabla u(\rho x+y)\cdot \frac{x}{\lvert x\rvert}\bigg\rvert\, d\mathcal{L}^1(\rho)\, d\mathscr{H}^{n-1}(x)\\
            &=\liminf_{i\to+\infty}\int_{\sigma_i}^{\rho_i}\int_{\mathbb{S}^{n-1}}\bigg\lvert\nabla u(\rho x+y)\cdot \frac{x}{\lvert x\rvert}\bigg\rvert\ d\mathscr{H}^{n-1}(x)\, d\mathcal{L}^1(\rho)\, \\
            &=\liminf_{i\to+\infty}\int_{\sigma_i}^{\rho_i}\int_{\partial B_{\rho}(y)}\frac{1}{\rho^{n-1}}\bigg\lvert\nabla u(z)\cdot \frac{z-y}{\lvert z-y\rvert}\bigg\rvert\ d\mathscr{H}^{n-1}(z)\, d\mathcal{L}^1(\rho)\\
            &=\liminf_{i\to+\infty}\int_{B_{\rho_i}(y)\smallsetminus B_{\sigma_i}(y)}\frac{1}{\lvert\,\cdot\,-y\rvert^{n-1}}\bigg\lvert\frac{\partial u}{\partial\nu_y}\bigg\rvert\, d\mathcal{L}^n\\
            &\le\liminf_{i\to+\infty}\int_0^{\rho_i}\frac{\sqrt{\phi(t)}}{t}\,d\mathcal{L}^1(t)=0,
        \end{align*}
    where the last equality follows from the assumption \eqref{Equation: integrability condition around the origin}. Hence, we have $\varphi_1|_{\mathbb{S}^{n-1}}=\varphi_2|_{\mathbb{S}^{n-1}}$. Since both $\varphi_1$ and $\varphi_2$ are $0$-homogeneous functions, we conclude that $\varphi_1=\varphi_2$ and the statement follows. 
    \end{proof}

\section{Epiperimetric and \L ojasiewicz--Simon inequalities} \label{appendix: EP ineq and LS ineq}
One of the strengths of the approach that we use in the proof of the epiperimetric equality for harmonic maps (see Theorem \ref{Theorem: (log)-epiperimetric inequality for harmonic maps}) is that we don't need to use any ``infinite dimensional'' \L ojasiewicz inequality. Indeed, by exploiting the Lyapunov--Schmidt reduction for the Dirichlet energy we can reduce to run a heat flow on the \underline{finite dimensional} kernel of the second variation of the energy and then leverage on the standard \L ojasiewicz inequality for real-analytic functionals on $\R^k$ (see Lemma \ref{lemma: finite dimensional Lojasiewicz inequality}).

In the framework of harmonic maps though, the following infinite dimensional \L ojasiewicz inequality was obtained by Simon in \cite{AsymptoticSimon} (see also \cite[Section 3.14]{simon-book} for a proof of this statement).
\begin{proposition}[\L ojasiewicz--Simon inequality for the Dirichlet energy on the sphere]\label{Proposition: Lojasiewicz-Simon inequality}
    Let $N\subset\R^k$ be a closed real-analytic submanifold in $\R^k$ and let $n\in\N$ be such that $n\ge 3$. Let $u_0\in C^{\infty}(\mathbb{S}^{n-1},N)$ be a harmonic map on $\mathbb{S}^{n-1}$. Then, there exist $\delta,C>0$ and $\beta\in(0,\frac{1}{2}]$ depending on $u_0$ such that 
    \begin{align*}
        \lvert\mathscr{E}_{\mathbb{S}^{n-1}}(u\,;u_0)\rvert^{1-\beta}\le C\|\nabla\mathscr{E}_{\mathbb{S}^{n-1}}(u\,;u_0)\|_{L^2(\mathbb{S}^{n-1})} \qquad\forall u\in\mathcal{B}_{\delta}(u_0)\subset C^{2,\alpha}(u_0^*TN).
    \end{align*}
\end{proposition}
Notice that, as suggested by \cite[Proposition 3.1]{colombo-spolaor-velichkov-3}, by exploiting the previous Proposition \eqref{Proposition: Lojasiewicz-Simon inequality} we can significantly simplify the proof of Theorem \ref{Theorem: (log)-epiperimetric inequality for harmonic maps} as follows. As by the assumptions of Theorem \ref{Theorem: (log)-epiperimetric inequality for harmonic maps}, let $N\subset\R^k$ be a closed real-analytic submanifold in $\R^k$ and let $n\in\N$ be such that $n\ge 3$. Let $u_0\in C^{\infty}(\mathbb{S}^{n-1},N)$ be a harmonic map on $\mathbb{S}^{n-1}$ and let $C,\delta>0$, $\beta\in(0,\frac{1}{2}]$ depending on $u_0$ be the constants given by Proposition \ref{Proposition: Lojasiewicz-Simon inequality}. Let $u\in\mathcal{B}_{\delta}(u_0)$ and notice that if 
\begin{align*}
    \mathscr{E}_{\mathbb{B}^n}(\tilde u\,;\tilde u_0)\le 0
\end{align*}
then choosing $\hat u:=\tilde u$ trivially gives \eqref{equation: main epiperimetric inequality}. Hence, we can assume that 
\begin{align*}
    \mathscr{E}_{\mathbb{B}^n}(\tilde u\,;\tilde u_0)=\frac{1}{n-2}\mathscr{E}_{\mathbb{S}^{n-1}}(u\,;u_0)>0,
\end{align*}
where the first equality follows from Lemma \ref{Lemma: slicing lemma}. By the existence of short-time smooth solutions for the harmonic map heat flow established in \cite{eells-sampson} (see also \cite[Theorem 5.2.1]{fanghua-lin-book}), there exists $0<\eps_0<\frac{1}{n-2}$ such that $U:[0,\eps_0]\to C^{\infty}(\mathbb{S}^{n-1},N)$ is a smooth solution of 
\begin{align*}
\begin{cases}
    U'(t)+\nabla\mathscr{D}_{\mathbb{S}^{n-1}}(U(t))=0\\
    U(0)=u
\end{cases}
\end{align*}
Moreover, by possibly choosing $\eps_0$ smaller, we can assume that 
\begin{align}\label{Equation: Lojasiewicz-Simon 2}
    \mathscr{E}_{\mathbb{S}^{n-1}}(u\,;u_0)\le 2\mathscr{E}_{\mathbb{S}^{n-1}}(U(t)\,;u_0) \qquad\forall\,t\in[0,\eps_0]
\end{align}
with equality just if $t=\eps_0$ (notice that $t\mapsto\mathscr{E}_{\mathbb{S}^{n-1}}(U(t)\,;u_0)$ is continuous and non-increasing in $t$). We let $\tilde U:[0,+\infty)\to C^{\infty}(\mathbb{S}^{n-1},N)$ be the given by
\begin{align*}
    \tilde U(t):=\begin{cases}
        U(t) & \mbox{ for every } 0\le t\le\eps_0\\
        U(\eps_0) & \mbox{ for every } t>\eps_0.
    \end{cases}
\end{align*}
Then we define the competitor $\hat u\in W^{1,2}(\mathbb{B}^n,N)$ in polar coordinates $(\rho,\theta)\in(0,+\infty)\times\mathbb{S}^{n-1}$ by 
\begin{align*}
    \hat u(\rho,\theta):=\tilde U(-\eps_0\log(\rho),\theta) \qquad\forall\,(\rho,\theta)\in(0,+\infty)\times\mathbb{S}^{n-1}.
\end{align*}
By Lemma \ref{Lemma: slicing lemma}, we have 
\allowdisplaybreaks
\begin{align}\label{Equation: Lojasiewicz-Simon 1}
    \nonumber
    \mathscr{E}_{\mathbb{B}^n}(\hat u\,;\tilde u)&=\int_0^1\mathscr{E}_{\mathbb{S}^{n-1}}(\hat u(\rho,\,\cdot\,)\,;u)\rho^{n-3}\, d\mathcal{L}^1(\rho)+\int_0^1\int_{\mathbb{S}^{n-1}}\lvert\partial_{\rho}\hat u(\rho,\theta)\rvert^2\, d\mathscr{H}^{n-1}(\theta)\rho^{n-1}\, d\mathcal{L}^1(\rho)\\
    \nonumber
    &=\int_0^1\mathscr{E}_{\mathbb{S}^{n-1}}(\tilde U(-\eps_0\log(\rho),\,\cdot\,)\,;u)\rho^{n-3}\, d\mathcal{L}^1(\rho)\\
    \nonumber
    &\quad+\eps_0^2\int_0^1\int_{\mathbb{S}^{n-1}}\lvert\tilde U'(-\eps_0\log(\rho),\theta)\rvert^2\, d\mathscr{H}^{n-1}(\theta)\rho^{n-3}\, d\mathcal{L}^1(\rho)\\
    \nonumber
    &=\frac{1}{\eps_0}\int_0^{+\infty}\mathscr{E}_{\mathbb{S}^{n-1}}(\tilde U(t,\,\cdot\,)\,;u)e^{-\frac{t(n-2)}{\eps_0}}\, d\mathcal{L}^1(t)\\
    \nonumber
    &\quad+\eps_0\int_0^{+\infty}\int_{\mathbb{S}^{n-1}}\lvert\tilde U'(t,\theta)\rvert^2\, d\mathscr{H}^{n-1}(\theta)e^{-\frac{t(n-2)}{\eps_0}}\, d\mathcal{L}^1(t)\\
    \nonumber
    &=\frac{1}{\eps_0}\int_0^{+\infty}\big(\mathscr{D}_{\mathbb{S}^{n-1}}(\tilde U(t,\,\cdot\,))-\mathscr{D}_{\mathbb{S}^{n-1}}(\tilde U(0,\,\cdot\,))\big)e^{-\frac{t(n-2)}{\eps_0}}\, d\mathcal{L}^1(t)\\
    \nonumber
    &\quad+\eps_0\int_0^{\eps_0}\int_{\mathbb{S}^{n-1}}\lvert U'(t,\theta)\rvert^2\, d\mathscr{H}^{n-1}(\theta)e^{-\frac{t(n-2)}{\eps_0}}\, d\mathcal{L}^1(t)\\
    &=\frac{1}{\eps_0}\int_0^{+\infty}\int_{0}^{\min\{t,\eps_0\}}\nabla\mathscr{D}_{\mathbb{S}^{n-1}}(U(s,\,\cdot\,))[U'(s,\,\cdot\,)]\, d\mathcal{L}^1(s)e^{-\frac{t(n-2)}{\eps_0}}\, d\mathcal{L}^1(t)\\
    \nonumber
    &\quad+\eps_0\int_0^{\eps_0}\| U'(t)\|_{L^2(\mathbb{S}^{n-1})}^2e^{-\frac{t(n-2)}{\eps_0}}\, d\mathcal{L}^1(t)\\
    \nonumber
    &=\frac{1}{\eps_0}\int_0^{\eps_0}\nabla\mathscr{D}_{\mathbb{S}^{n-1}}(U(s,\,\cdot\,))[U'(s,\,\cdot\,)]\bigg(\int_{s}^{+\infty}e^{-\frac{t(n-2)}{\eps_0}}\, d\mathcal{L}^1(t)\bigg)\, d\mathcal{L}^1(s)\\
    \nonumber
    &\quad+\eps_0\int_0^{\eps_0}\|U'(t)\|_{L^2(\mathbb{S}^{n-1})}^2e^{-\frac{t(n-2)}{\eps_0}}\, d\mathcal{L}^1(t)\\
    \nonumber
    &=-\frac{1}{\eps_0}\int_0^{\eps_0}\hspace{-1mm}\bigg(\hspace{-1.5mm}-\frac{\eps_0}{n-2}\nabla\mathscr{D}_{\mathbb{S}^{n-1}}(U(s,\,\cdot\,))[U'(s,\,\cdot\,)]-\eps_0^2\|U'(s)\|_{L^2(\mathbb{S}^{n-1})}^2\bigg)e^{-\frac{s(n-2)}{\eps_0}}\, d\mathcal{L}^1(s)\\
    \nonumber
    &=-\frac{1}{\eps_0}\int_0^{\eps_0}\bigg(\frac{\eps_0}{n-2}-\eps_0^2\bigg)\|U'(s)\|_{L^2(\mathbb{S}^{n-1})}^2\bigg)e^{-\frac{s(n-2)}{\eps_0}}\, d\mathcal{L}^1(s)\\
    \nonumber
    &=-\tilde C\int_0^{\eps_0}\|\nabla\mathscr{D}_{\mathbb{S}^{n-1}}(U(s))\|_{L^2(\mathbb{S}^{n-1})}^2e^{-\frac{s(n-2)}{\eps_0}}\, d\mathcal{L}^1(s),
\end{align}
where we have let
\begin{align*}
    \tilde C:=\frac{1}{\eps_0}\bigg(\frac{\eps_0}{n-2}-\eps_0^2\bigg)>0.
\end{align*}
Now let $\tilde\eps\in(0,1)$ be a small constant to be chosen later and notice that, by \eqref{Equation: Lojasiewicz-Simon 1}, by Proposition \ref{Proposition: Lojasiewicz-Simon inequality} and by \eqref{Equation: Lojasiewicz-Simon 2} we have that
\begin{align*}
    \mathscr{E}_{\mathbb{B}^n}(\hat u\,;\tilde u_0)&-(1-\tilde\eps)\mathscr{E}_{\mathbb{B}^n}(\tilde u\,;\tilde u_0)\\
    &=-\tilde C\int_0^{\eps_0}\|\nabla\mathscr{D}_{\mathbb{S}^{n-1}}(U(s))\|_{L^2(\mathbb{S}^{n-1})}^2e^{-\frac{s(n-2)}{\eps_0}}\, d\mathcal{L}^1(s)+\tilde\eps\mathscr{E}_{\mathbb{B}^{n}}(\tilde u\,;\tilde u_0)\\
    &\le -\tilde CC^2\int_0^{\eps_0}\mathscr{E}_{\mathbb{S}^{n-1}}(U(s)\,; u_0)^{2-2\beta}e^{-\frac{s(n-2)}{\eps_0}}\, d\mathcal{L}^1(s)+\tilde\eps\mathscr{E}_{\mathbb{B}^{n}}(\tilde u\,;\tilde u_0)\\
    &\le -2^{2\beta-2}\tilde CC^2\int_0^{\eps_0}\mathscr{E}_{\mathbb{S}^{n-1}}(u\,;u_0)^{2-2\beta}e^{-\frac{s(n-2)}{\eps_0}}\, d\mathcal{L}^1(s)+\tilde\eps\mathscr{E}_{\mathbb{B}^{n}}(\tilde u\,;\tilde u_0)\\
    &\le -\bigg(\frac{n-2}{2}\bigg)^{2-2\beta}\tilde CC^2\int_0^{\eps_0}\mathscr{E}_{\mathbb{B}^{n}}(\tilde u\,;\tilde u_0)^{2-2\beta}e^{-\frac{s(n-2)}{\eps_0}}\, d\mathcal{L}^1(s)+\tilde\eps\mathscr{E}_{\mathbb{B}^{n}}(\tilde u\,;\tilde u_0)\\
    &\le-\bigg(\bigg(\frac{n-2}{2}\bigg)^{1-2\beta}\eps_0\tilde CC^2\big(1-e^{-(n-2)}\big)-\tilde\eps\mathscr{E}_{\mathbb{B}^{n}}(\tilde u\,;\tilde u_0)^{2\beta-1}\bigg)\mathscr{E}_{\mathbb{B}^n}(\tilde u\,;\tilde u_0)^{2-2\beta}.
\end{align*}
Now choosing $\tilde\eps:=\eps\mathscr{E}_{\mathbb{S}^{n-1}}(u\,;\tilde u_0)^{1-2\beta}$ in the previous inequality for some $\eps>0$ small enough depending on $u_0$ and $n$ we get 
\begin{align*}
    \mathscr{E}_{\mathbb{B}^n}(\hat u\,;\tilde u_0)&-(1-\eps\mathscr{E}_{\mathbb{B}^{n}}(\tilde u\,;\tilde u_0)^{1-2\beta})\mathscr{E}_{\mathbb{B}^n}(\tilde u\,;\tilde u_0)\le 0
\end{align*}
and, since we assumed that $\mathscr{E}_{\mathbb{B}^n}(\tilde u\,;\tilde u_0)>0$, the statement follows with $\gamma:=2\beta$.  
\bibliographystyle{amsalpha} 
\bibliography{references} 

\end{document}